\documentclass[20pt]{article}
\usepackage{fourier}
\usepackage[]{geometry}
\usepackage{mathrsfs}
\usepackage{amsfonts}
\usepackage{amssymb}
\usepackage{amsmath}
\usepackage{enumerate}
\usepackage[authoryear]{natbib}
\usepackage[pdftex]{color,graphicx}
\newcommand{\function}[2]{#1\left(#2\right)}
\newcommand{\p}[1]{\mathbf{P}\left(#1\right)}

\newcommand{\ip}[1]{\mathbf{P}_*\left(#1\right)}
\newcommand{\op}[1]{\mathbf{P}^*\left(#1\right)}
\newcommand{\e}[1]{\mathbf{E}\left(#1\right)}
\newcommand{\var}[1]{\textrm{Var}\left(#1\right)}
\newenvironment{proof}{\textbf{Proof: }}{$\hfill \square \\$}

\newtheorem{lemma}{Lemma}[section]

\newtheorem{thm}{Theorem}[section]
\newcommand{\ind}[1]{\mathbf{1}_{#1}}
\newcommand{\cas}{\stackrel{a.s.}{\longrightarrow}}

\newcommand{\argmax}{\operatornamewithlimits{\textrm{argmax}}}
\newcommand{\argmin}{\operatornamewithlimits{\textrm{argmin}}}
\newcommand{\lsup}{\operatornamewithlimits{\overline{\lim}}}
\newcommand{\linf}{\operatornamewithlimits{\underline{\lim}}}

\newcommand{\conv}[1]{Conv\left(#1\right)}
\usepackage[colorlinks,citecolor=blue,urlcolor=blue,linkcolor=magenta,naturalnames=true]{hyperref}
\begin{document}
\title{\textsc{Nonparametric Least Squares Estimation of a Multivariate Convex Regression Function}}
\author{\begin{tabular}{c}
Emilio Seijo and Bodhisattva Sen\\
Columbia University
\end{tabular}}
\date{July 10, 2010}
\maketitle
\fontsize{12}{14pt}
\selectfont

\begin{abstract}
{This paper deals with the consistency of the least squares estimator of a convex regression function when the predictor is multidimensional. We characterize and discuss the computation of such an estimator via the solution of certain quadratic and linear programs. Mild sufficient conditions for the consistency of this estimator and its subdifferentials in fixed and stochastic design regression settings are provided. We also consider a regression function which is known to be convex and componentwise nonincreasing and discuss the characterization, computation and consistency of its least squares estimator.}
\end{abstract}


\section{Introduction}
Consider a closed, convex set $\mathfrak{X}\subset\mathbb{R}^d$, for $d \ge 1$, with nonempty interior and a regression model of the form
\begin{equation}\label{ec1}
Y=\phi(X) + \epsilon
\end{equation}
where $X$ is a $\mathfrak{X}$-valued random vector, $\epsilon$ is a random variable with $\e{\epsilon\left|X\right.}=0$, and $\phi:\mathbb{R}^d\rightarrow\mathbb{R}$ is an unknown {\it convex} function. Given independent observations $(X_1,Y_1),\ldots,(X_n,Y_n)$ from such a model, we wish to estimate $\phi$ by the method of least squares, i.e., by finding a convex function $\hat{\phi}_{n}$ which minimizes the discrete $\mathcal{L}_2$ norm
\[ \left(\sum_{k=1}^n \left|Y_k - \psi(X_k)\right|^2\right)^{\frac{1}{2}}\]
among all convex functions $\psi$ defined on the convex hull of $X_1,\ldots,X_n$. In this paper we characterize the least squares estimator, provide means for its computation, study its finite sample properties and prove its consistency.

The problem just described is a nonparametric regression problem with known shape restriction (convexity). Such problems have a long history in the statistical literature with seminal papers like \cite{bru}, \cite{gre} and \cite{hil} written more than 50 years ago, albeit in simpler settings. The former two papers deal with the estimation of monotone functions while the latter discusses least squares estimation of a concave function whose domain is a subset of the real line. Since then, many results on different nonparametric shape restricted regression problems have been published. For instance, \cite{bru70} and, more recently, \cite{zhang02} have enriched the literature concerning isotonic regression. In the particular case of convex regression, \cite{hp76} proved the consistency of the least squares estimator introduced in \cite{hil}. Some years later, \cite{mam} and \cite{gjw} derived, respectively, the rate of convergence and asymptotic distribution of this estimator. Some alternative methods of estimation that combine shape restrictions with smoothness assumptions have also been proposed for the one-dimensional case; see, for example, \cite{biho} where a kernel-based estimator is defined and its asymptotic distribution derived.

Although the asymptotic theory of the one-dimensional convex regression problem is well understood, not much has been done in the multidimensional scenario. The absence of a natural order structure in $\mathbb{R}^d$, for $d > 1$, poses a natural impediment in such extensions. A convex function on the real line can be characterized as an absolutely continuous function with increasing first derivative (see, for instance, \cite{fol}, Exercise 42.b, page 109). This characterization plays a key role in the computation and asymptotic theory of the least squares estimator in the one-dimensional case. By contrast, analogous results for convex functions of several variables involve more complicated characterizations using either second-order conditions (as in \cite{du77}, Theorem 3.1, page 163) or cyclical monotonicity (as in \cite{rca}, Theorems 24.8 and 24.9, pages 238-239). Interesting differences between convex functions on $\mathbb{R}$ and convex functions on $\mathbb{R}^d$ are given in \cite{johan74} and \cite{bro78a}.

Recently there has been considerable interest in shape restricted function estimation in multidimension. In the density estimation context, \cite{cusas08} deal with the computation of the nonparametric maximum likelihood estimator of a multidimensional log-concave density, while \cite{cusa09}, \cite{schudu} and \cite{scdu} discuss its consistency and related issues. \cite{sewe} study the computation and consistency of the maximum likelihood estimator of convex-transformed densities. This paper focuses on estimating a regression function which is known to be convex. To the best of our knowledge this is the first attempt to systematically study the characterization, computation, and consistency of the least squares estimator of a convex regression function with multidimensional covariates in a {\it completely nonparametric} setting.

In the field of econometrics some work has been done on this multidimensional problem in less general contexts and with more stringent assumptions. Estimation of concave and/or componentwise nondecreasing functions has been treated, for instance, in \cite{bama}, \cite{mat1}, \cite{mat2}, \cite{bersieve} and \cite{necpt}. The first two papers define maximum likelihood estimators in semiparametric settings. The estimators in \cite{mat1} and \cite{bama} are shown to be consistent in \cite{mat1} and \cite{masa}, respectively. A maximum likelihood estimator and a sieved least squares estimator have been defined and techniques for their computation have been provided in \cite{necpt} and \cite{bersieve}, respectively.

The method of least squares has been applied to multidimensional concave regression in \cite{kuo}. We take this work as our starting point. In agreement with the techniques used there, we define a least squares estimator which can be computed by solving a quadratic program. We argue that this estimator can be evaluated at a single point by finding the solution to a linear program. We then show that, under some mild regularity conditions, our estimator can be used to consistently estimate both, the convex function and its subdifferentials.

Our work goes beyond those mentioned above in the following ways: Our method does not require any tuning parameter(s), which is a major drawback for most nonparametric regression methods, such as kernel-based procedures. The choice of the tuning parameter(s) is especially problematic in higher dimensions, e.g., kernel based methods would require the choice of a $d \times d$ matrix of bandwidths. The sets of assumptions that most authors have used to study the estimation of a multidimensional convex regression function are more restrictive and of a different nature than the ones in this paper. As opposed to the maximum likelihood approach used in \cite{bama}, \cite{mat1}, \cite{necpt} and \cite{masa}, we prove the consistency of the estimator keeping the distribution of the errors {\it unspecified}; e.g., in the i.i.d. case we only assume that the errors have zero expectation and finite second moment. The estimators in \cite{bersieve} are sieved least squares estimators and assume that the observed values of the predictors lie on equidistant grids of rectangular domains. By contrast, our estimators are unsieved and our assumptions on the spatial arrangement of the predictor values are much more relaxed. In fact, we prove the consistency of the least squares estimator under both fixed and stochastic design settings; we also allow for heteroscedastic errors. In addition, we show that the least squares estimator can also be used to approximate the gradients and subdifferentials of the underlying convex function.

It is hard to overstate the importance of convex functions in applied mathematics. For instance, optimization problems with convex objective functions over convex sets appear in many applications. Thus, the question of accurately estimating a convex regression function is indeed interesting from a theoretical perspective. However, it turns out that convex regression is important for numerous reasons besides statistical curiosity. Convexity also appears in many applied sciences. One such field of application is microeconomic theory. Production functions are often supposed to be concave and componentwise nondecreasing. In this context, concavity reflects decreasing marginal returns. Concavity also plays a role in the theory of rational choice since it is a common assumption for utility functions, on which it represents decreasing marginal utility. The interested reader can see \cite{hil}, \cite{varian1} or \cite{varian2} for more information regarding the importance of concavity/convexity in economic theory.

The paper is organized as follows. In Section \ref{sec1} we discuss the estimation procedure, characterize the estimator and show how it can be computed by solving a positive semidefinite quadratic program and a linear program. Section \ref{sec2} starts with a description of the deterministic and stochastic design regression schemes. The statement and proof of our main results are also included in Section \ref{sec2}. In Section \ref{sec4} we provide the proofs of the technical lemmas used to prove the main theorem. Section \ref{sec5}, the Appendix, contains some results from convex analysis and linear algebra that are used in the paper and may be of independent interest.

\section{Characterization and finite sample properties}\label{sec1}
We start with some notation. For convenience, we will regard elements of the Euclidian space $\mathbb{R}^m$ as column vectors and denote their components with upper indices, i.e, any $z \in \mathbb{R}^m$ will be denoted as $z=(z^1,z^2,\ldots,z^m)$. The symbol $\overline{\mathbb{R}}$ will stand for the extended real line. Additionally, for any set $A\subset\mathbb{R}^d$ we will denoted as $\conv{A}$ its convex hull and we'll write $\conv{X_1,\ldots,X_n}$ instead of $\conv{\{X_1,\ldots,X_n\}}$. Finally, we will use $\langle\cdot,\cdot\rangle$ and $|\cdot|$ to denote the standard inner product and norm in Euclidian spaces, respectively.

For $\mathcal{X} =\{X_1,\ldots,X_n\} \subset\mathfrak{X} \subset\mathbb{R}^d$, consider the set $\mathcal{K}_\mathcal{X}$ of all vectors $z = (z^1,\ldots,z^n)' \in\mathbb{R}^n$ for which there is a convex function $\psi:\mathfrak{X}\rightarrow\mathbb{R}$ such that $\psi(X_j) = z^j$ for all $j=1,\ldots,n$. Then, a necessary and sufficient condition for a convex function $\psi$ to minimize the sum of squared errors is that $\psi(X_j) = Z_n^j$ for $j=1,\ldots,n$, where
\begin{equation}
Z_n = \argmin_{z\in\mathcal{K}_\mathcal{X}}\left\{ \sum_{k=1}^n \left|Y_k - z^k\right|^2\right\}. \label{ec40}
\end{equation}

The computation of the vector $Z_n$ is crucial for the estimation procedure. We will show that such a vector exists and is unique. However, it should be noted that there are many convex functions $\psi$ satisfying $\psi(X_j) = Z_n^j$ for all $j=1,\ldots,n$. Although any of these functions can play the role of the least squares estimator, there is one such function which is easily evaluated in $\conv{X_1,\ldots,X_n}$. For computational convenience, we will define our least squares estimator $\hat{\phi}_n$ to be precisely this function and describe it explicitly in (\ref{ec6}) and the subsequent discussion.

In what follows we show that both, the vector $Z_n$ and the least squares estimator $\hat{\phi}_n$ are well-defined for any $n$ data points $(X_1,Y_1),\ldots,(X_n,Y_n)$. We will also provide two characterizations of the set $\mathcal{K}_\mathcal{X}$ and show that the vector $Z_n$ can be computed by solving a positive semidefinite quadratic program. Finally, we will prove that for any $x\in\conv{X_1,\ldots,X_n}$ one can obtain $\hat{\phi}_n(x)$ by solving a linear program.

\subsection{Existence and uniqueness}
We start with two characterizations of the set $\mathcal{K}_\mathcal{X}$. The developments here are similar to those in \cite{necpt} and \cite{kuo}.

\begin{lemma}[Primal Characterization]\label{l2}
Let $z=(z^1,\ldots,z^n)\in\mathbb{R}^n$. Then, $z\in\mathcal{K}_\mathcal{X}$ if and only if for every $j=1,\ldots,n$, the following holds:
\begin{equation}\label{ec3}
z^j = \inf\left\{ \sum_{k=1}^n \theta^k z^k : \sum_{k=1}^n \theta^k = 1,\ \sum_{k=1}^n \theta^k X_k = X_j,\ \theta\geq 0,\ \theta\in\mathbb{R}^n \right\},
\end{equation}
where the inequality $\theta\geq 0 $ holds componentwise.
\end{lemma}
\begin{proof} Define the function $g:\mathbb{R}^d\rightarrow\overline{\mathbb{R}}$ by
\begin{equation}\label{ec2}
g(x) = \inf\left\{ \sum_{k=1}^n \theta^k z^k : \sum_{k=1}^n \theta^k = 1,\ \sum_{k=1}^n \theta^k X_k = x,\ \theta\geq 0,\ \theta\in\mathbb{R}^n \right\}
\end{equation}
where we use the convention that $\inf(\emptyset)=+\infty$. By Lemma \ref{l1} in the Appendix, $g$ is convex and finite on the $X_j$'s. Hence, if $z^j$ satisfies (\ref{ec3}) then $z^j = g(X_j)$ for every $j=1,\ldots,n$ and it follows that $z\in\mathcal{K}_\mathcal{X}$.

Conversely, assume that $z\in\mathcal{K}_\mathcal{X}$ and $g(X_j) \neq z^j$ for some $j$. Note that $g(X_k)\leq z^k$ for any $k$ from the definition of $g$. Thus, we may suppose that $g(X_j) < z^j$. As $z\in\mathcal{K}_\mathcal{X}$, there is a convex function $\psi$ such that $\psi(X_k) = z^k$ for all $k=1,\ldots,n$. Then, from the definition of $g(X_j)$ there exist $\theta_0\in\mathbb{R}^n$ with $\theta_0\geq 0$ and $\theta_0^1 + \ldots + \theta_0^n = 1$ such that $\theta_0^1 X_1 + \ldots + \theta_0^n X_n = X_j$ and
\[ \sum_{k=1}^n \theta_0^k \psi(X_k) = \sum_{k=1}^n \theta_0^k z^k < z^j = \psi (X_j) = \function{\psi}{\sum_{k=1}^n \theta_0^k X_k}, \]
which leads to a contradiction because $\psi$ is convex.
\end{proof}

We now provide an alternative characterization of the set $\mathcal{K}_\mathcal{X}$ based on the dual problem to the linear program used in Lemma \ref{l2}.

\begin{lemma}[Dual Characterization]\label{l3}
Let $z\in\mathbb{R}^n$. Then, $z\in\mathcal{K}_\mathcal{X}$ if and only if for any $j=1,\ldots,n$ we have
{\small \begin{equation}\label{ec4}
z^j = \sup\left\{ \langle\xi,X_j\rangle  + \eta : \langle\xi,X_k\rangle + \eta \leq z^k\ \ \forall\ k=1,\ldots,n,\ \xi\in\mathbb{R}^d,\ \eta\in\mathbb{R}\right\}.
\end{equation}}
Moreover, $z\in\mathcal{K}_\mathcal{X}$ if and only if there exist vectors $\xi_1,\ldots,\xi_n\in\mathbb{R}^d$ such that
\begin{equation}\label{ec5}
\langle \xi_j, X_k-X_j \rangle \leq z^k - z^j\ \ \forall\ k,j\in\{1,\ldots,n\}.
\end{equation}
\end{lemma}
\begin{proof} According to the primal characterization, $z\in\mathcal{K}_\mathcal{X}$ if and only if the linear programs defined by (\ref{ec3}) have the $z^j$'s as optimal values. The linear programs in (\ref{ec4}) are the dual problems to those in (\ref{ec3}). Then, the duality theorem for linear programs (see \cite{lbg}, page 89) implies that the $z^j$'s have to be the corresponding optimal values to the programs in (\ref{ec4}).

To prove the second assertion let us first assume that $z\in\mathcal{K}_\mathcal{X}$. For each $j\in\{1,\ldots,n\}$ take any solution $(\xi_j,\eta_j)$ to (\ref{ec4}). Then by (\ref{ec4}), $\eta_j = z^j - \langle \xi_j, X_j \rangle$ and the inequalities in (\ref{ec5}) follow immediately because we must have $ \langle\xi_j,X_k\rangle + \eta_j \leq z^k$ for any $k\in\{1,\ldots,n\}$. Conversely, take $z\in\mathbb{R}^n$ and assume that there are $\xi_1,\ldots,\xi_n\in\mathbb{R}^d$ satisfying (\ref{ec5}). Take any $j\in\{1,\ldots,n\}$, $\eta_j = z^j - \langle \xi_j,X_j \rangle$ and $\theta$ to be the vector in $\mathbb{R}^n$ with components $\theta^k = \delta_{kj}$, where $\delta_{kj}$ is the Kronecker $\delta$. It follows that $\langle \xi_j, X_k\rangle + \eta_j \leq z^k$ $\forall$ $k=1,\ldots, n$ so $(\xi_j,\eta_j)$ is feasible for the linear program in (\ref{ec4}). In addition, $\theta$ is feasible for the linear program in (\ref{ec3}) so the weak duality principle of linear programming (see \cite{lbg}, Lemma 1, page 89) implies that $\langle \xi ,X_j \rangle + \eta \leq z^j$ for any pair $(\xi,\eta)$ which is feasible for the problem in the right-hand side of (\ref{ec4}). We thus have that $z^j$ is an upper bound attained by the feasible pair $(\xi_j,\eta_j)$ and hence (\ref{ec4}) holds for all $j=1,\ldots,n$.
\end{proof}

Both, the primal and dual characterizations are useful for our purposes. The primal plays a key role in proving the existence and uniqueness of the least squares estimator. The dual is crucial for its computation.

\begin{lemma}\label{l4}
The set $\mathcal{K}_\mathcal{X}$ is a closed, convex cone in $\mathbb{R}^n$ and the vector $Z_n$ satisfying (\ref{ec40}) is uniquely defined.
\end{lemma}
\begin{proof} That $\mathcal{K}_\mathcal{X}$ is a convex cone follows trivially from the definition of the set. Now, if $z\notin \mathcal{K}_\mathcal{X}$, then there is $j\in\{1,\ldots,n\}$ for which
$z^j > g(X_j)$ with the function $g$ defined as in (\ref{ec2}). Thus, there is $\theta_0\in\mathbb{R}^n$ with $\theta_0\geq 0$ and $\theta_0^1 + \ldots + \theta_0^n = 1$ such that $\theta_0^1 X_1 + \ldots + \theta_0^n X_n = X_j$ and $\sum_{k=1}^n \theta_0^k z^k < z^j$. Setting $ \delta = \frac{1}{2}\left(z^j - \sum_{k=1}^n \theta_0^k z^k\right)$ it is easily seen that for all $ \zeta\in\prod_{k=1}^n (z^k-\delta,z^k + \delta)$ we still have $ \sum_{k=1}^n \theta_0^k \zeta^k < \zeta^j$ and thus $\zeta \notin \mathcal{K}_\mathcal{X}$. Therefore we have shown that for any $z \notin \mathcal{K}_\mathcal{X}$ there is a neighborhood $U$ of $z$ with $U\subset \mathbb{R}^n\setminus\mathcal{K}_\mathcal{X}$. Therefore, $\mathcal{K}_\mathcal{X}$ is closed and the vector $Z_n$ is uniquely determined as the projection of $(Y_1,\ldots,Y_n)\in\mathbb{R}^n$ onto the closed convex set $\mathcal{K}_\mathcal{X}$ (see \cite{con}, Theorem 2.5, page 9).
\end{proof}

We are now in a position to define the least squares estimator. Given observations $(X_1,Y_1),\ldots,(X_n,Y_n)$ from model (\ref{ec1}), we take the nonparametric least squares estimator to be the function $\hat{\phi}_n:\mathbb{R}^d\rightarrow\mathbb{R}$ defined by
\begin{equation}\label{ec6}
\function{\hat{\phi}_n}{x} = \inf\left\{ \sum_{k=1}^n \theta^k Z_n^k : \sum_{k=1}^n \theta^k = 1,\ \sum_{k=1}^n \theta^k X_k = x,\ \theta\geq 0,\ \theta\in\mathbb{R}^n \right\}
\end{equation}
for any $x\in\mathbb{R}^d$. Here we are taking the convention that $\inf(\emptyset) = +\infty$. This function is well-defined because the vector $Z_n$ exists and is unique for the sample. The estimator is, in fact, a polyhedral convex function (i.e., a convex function whose epigraph is a polyhedral; see \cite{rca}, page 172) and satisfies, as a consequence of Lemma \ref{l1},
\begin{equation}\nonumber
\hat{\phi}_n (x) = \sup_{\psi\in\mathcal{K}_{\mathcal{X},Z_n}}\{ \psi(x)\},
\end{equation}
where $\mathcal{K}_{\mathcal{X},Z_n}$ is the collection of all convex functions $\psi:\mathbb{R}^d\rightarrow\mathbb{R}$ such that $\psi(X_j) \leq Z_n^j$ for all $j=1,\ldots,n$. Thus, $\hat{\phi}_n$ is the largest convex function that never exceeds the $Z_n^j$'s. It is immediate that $\hat{\phi}_n$ is indeed a convex function (as the supremum of any family of convex functions is itself convex). The primal characterization of the set $\mathcal{K}_\mathcal{X}$ implies that $\hat{\phi}_n(X_j) = Z_n^j$ for all $j=1,\ldots,n$.

\subsection{Finite sample properties}
In the following lemma we state some of the most important finite sample properties of the least squares estimator defined by (\ref{ec6}).

\begin{lemma}\label{l6}
Let $\hat{\phi}_n$ be the least squares estimator obtained from the sample $(X_1,Y_1),\ldots,(X_n,Y_n)$. Then,
\begin{enumerate}[(i)]
\item $\displaystyle \sum_{k=1}^n (\psi(X_k) - \hat{\phi}_n(X_k))(Y_k - \hat{\phi}_n(X_k))\leq 0$ for any convex function $\psi$ which is finite on $\conv{X_1,\ldots,X_n}$;
\item $\displaystyle \sum_{k=1}^n \hat{\phi}_n(X_k)(Y_k - \hat{\phi}_n(X_k))= 0$;
\item $\displaystyle \sum_{k=1}^n Y_k = \sum_{k=1}^n \hat{\phi}_n(X_k)$;
\item the set on which $\hat{\phi}_n < \infty$ is $\conv{X_1,\ldots,X_n}$;
\item for any $x\in\mathbb{R}^d$ the map $(X_1,\ldots,X_n,Y_1,\ldots,Y_n)\hookrightarrow\hat{\phi}_n(x)$ is a Borel-measurable function from $\mathbb{R}^{n(d+1)}$ into $\mathbb{R}$.
\end{enumerate}
\end{lemma}
\begin{proof} Property $(i)$ follows from Moreau's decomposition theorem, which can be stated as: \newline
{\it Consider a closed convex set $\mathcal{C}$ on a Hilbert space $\mathcal{H}$ with inner product $\langle\cdot,\cdot\rangle$ and norm $\|\cdot\|$. Then, for any $x\in\mathcal{H}$ there is only one vector $x_\mathcal{C}\in\mathcal{C}$ satisfying $\|x - x_\mathcal{C}\| = \argmin_{\xi\in\mathcal{C}}\{\|x-\xi\|\}$. The vector $x_\mathcal{C}$ is characterized by being the only element of $\mathcal{C}$ for which the inequality $\langle \xi-x_\mathcal{C},x-x_\mathcal{C}\rangle\leq 0$ holds for every $\xi\in\mathcal{C}$ (see \cite{mor} or \cite{chinos}).}

Taking $\psi$ to be $\kappa \hat{\phi}_n$ and letting $\kappa$ vary through $(0,\infty)$ gives $(ii)$ from $(i)$. Similarly, $(iii)$ follows from $(i)$ by letting $\psi$ to be $\hat{\phi}_n \pm 1$. Property $(iv)$ is obvious from the definition of $\hat{\phi}_n$.

To see why $(v)$ holds, we first argue that the map $(X_1,\ldots,X_n,Y_1,$ $ \ldots,Y_n)\hookrightarrow Z_n$ is measurable. This follows from the fact that $Z_n$ is the solution to a convex quadratic program and thus can be found as a limit of sequences whose elements come from arithmetic operations with $(X_1,\ldots,X_n,$ $Y_1,\ldots,Y_n)$. Examples of such sequences are the ones produced by active set methods, e.g, see \cite{bolqp}; or by interior-point methods (see \cite{kavaqp} or \cite{mesuqp}). The measurability of $\hat{\phi}_n(x)$ follows from a similar argument, since it is the optimal value of a linear program whose solution can be obtained from arithmetic operations involving just $(X_1,\ldots,X_n,Y_1,\ldots,Y_n)$ and $Z_n$ (e.g., via the well-known simplex method; see \cite{nowr}, page 372 or \cite{lbg}, page 30).
\end{proof}

\subsection{Computation of the estimator}
Once the vector $Z_n$ defined in (\ref{ec40}) has been obtained, the evaluation of $\hat{\phi}_n$ at a single point $x$ can be carried out by solving the linear program in (\ref{ec6}). Thus, we need to find a way to compute $Z_n$. And here the dual characterization proves of vital importance, since it allows us to compute $Z_n$ by solving a quadratic program.
\begin{lemma}\label{l7}
Consider the positive semidefinite quadratic program
\begin{equation}\label{ec7}
\begin{array}{rc}
\min &  \sum_{k=1}^n |Y_k - z^k|^2 \\
\textrm{subject to} & \langle \xi_k , X_j - X_k \rangle \leq z^j - z^k\ \ \forall\ k,j=1,\ldots,n \\
 & \xi_1,\ldots,\xi_n\in\mathbb{R}^d, z\in\mathbb{R}^n.
\end{array}
\end{equation}
Then, this program has a unique solution $Z_n$ in $z$, i.e., for any two solutions $(\xi_1,\ldots,\xi_n,z)$ and $(\tau_1,\ldots,\tau_n,\zeta)$ we have $z = \zeta = Z_n$. This solution $Z_n$ is the only vector in $\mathbb{R}^n$ which satisfies (\ref{ec40}).
\end{lemma}
\begin{proof} From Lemma \ref{l3} if $(\xi_1,\ldots,\xi_n,z)$ belongs in the feasible set of this program, then $z\in\mathcal{K}_{\mathcal{X}}$. Moreover, for any $z\in\mathcal{K}_{\mathcal{X}}$ there are $\xi_1,\ldots,\xi_n\in\mathbb{R}^d$ such that $(\xi_1,\ldots,\xi_n,z)$ belongs to the feasible set of the quadratic program. Since the objective function only depends on $z$, solving the quadratic program is the same as getting the element of $\mathcal{K}_\mathcal{X}$ which is the closest to $Y$. This element is, of course, the uniquely defined $Z_n$ satisfying (\ref{ec40}).
\end{proof}
The quadratic program (\ref{ec7}) is positive semidefinite. This implies certain computational complexities, but most modern nonlinear programming solvers can handle this type of optimization problems. Some examples of high-performance quadratic programming solvers are CPLEX, LINDO, \\ MOSEK and QPOPT.
\begin{figure}[h!]
\centering
\includegraphics[height=6cm,width=6cm]{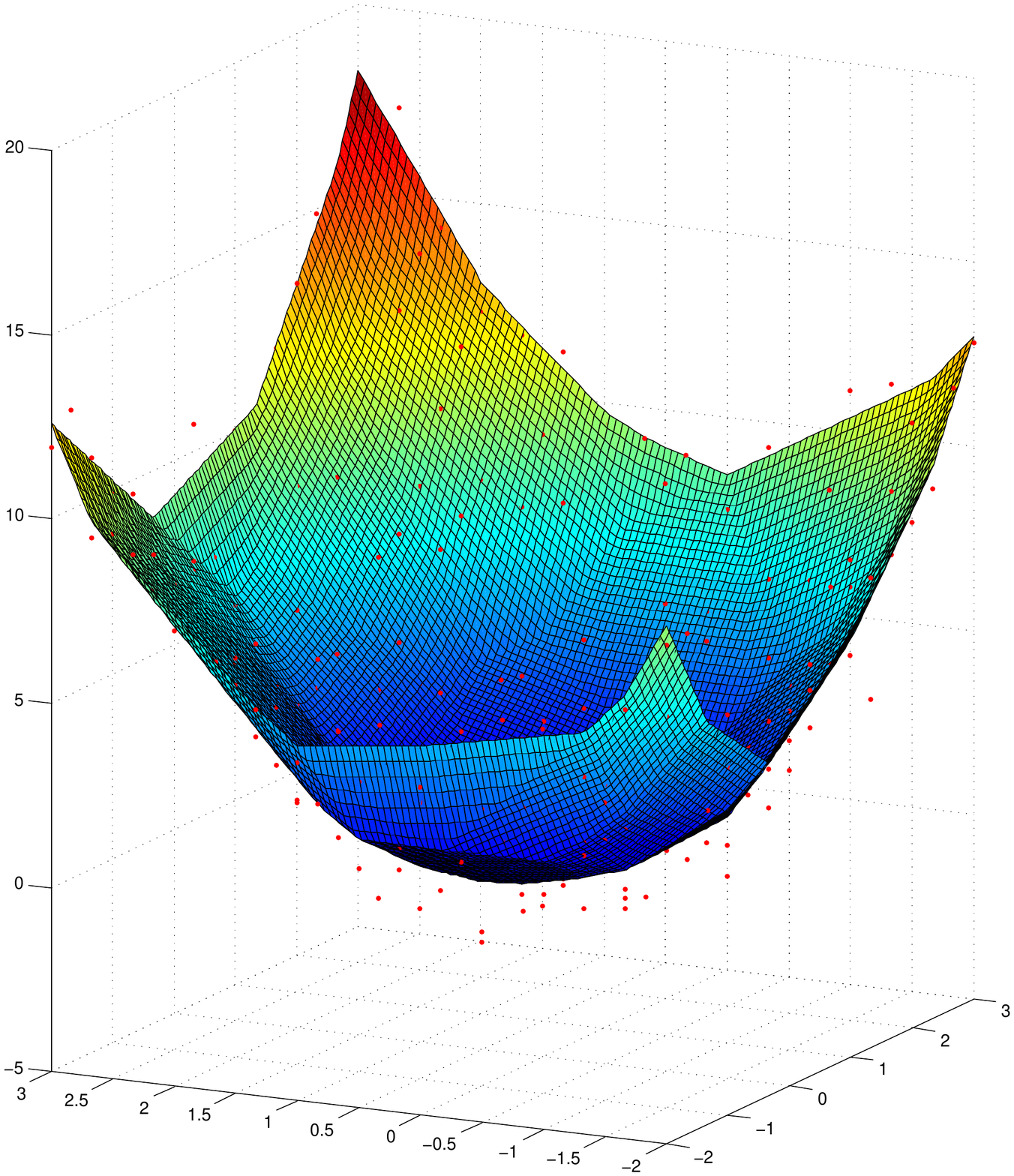} \includegraphics[height=6cm,width=6cm]{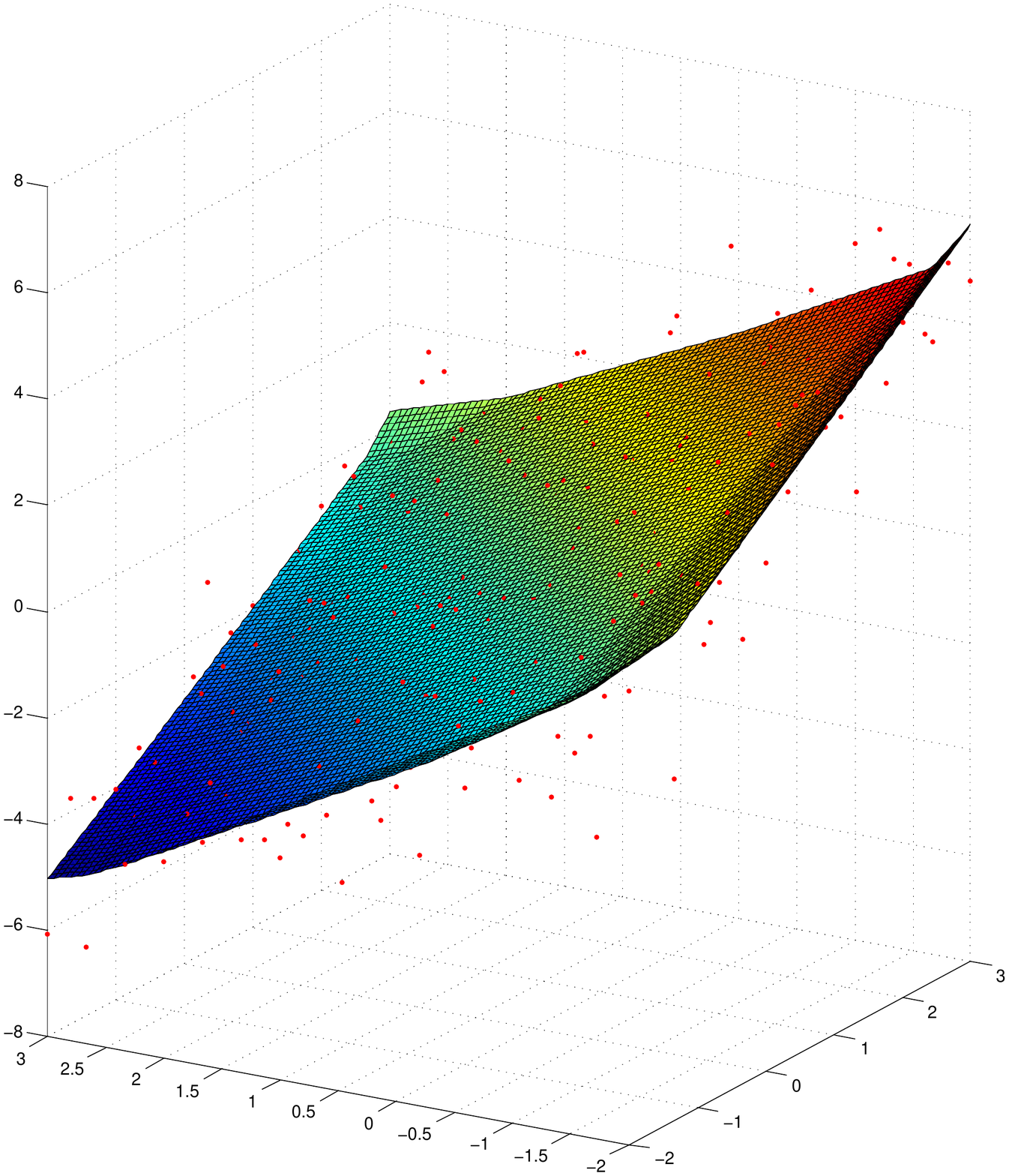}
\caption{The scatter plot and nonparametric least squares estimator of the convex regression function when (a) $\phi(x)=|x|^2$ (left panel); (b) $\phi(x) = -x^1 + x^2$ (right panel).}\label{figura6}
\end{figure}
Here we present two simulated examples to illustrate the computation of the estimator when $d=2$. The first one, depicted in Figure \ref{figura6}a corresponds to the case where $\phi(x) = |x|^2$. Figure \ref{figura6}b shows the convex function estimator when the regression function is the hyperplane $\phi(x) = -x^1 + x^2$. In both cases, $n=256$ observations were used and the errors were assumed to be i.i.d. from the standard normal distribution. All the computations were carried out using the MOSEK optimization toolbox for Matlab and the run time for each example was less than 2 minutes in a standard desktop PC. We refer the reader to \cite{kuo} for additional numerical examples (although the examples there are for the estimation of concave, componentwise nondecreasing functions, the computational complexities are the same).

\subsection{The componentwise nonincreasing case}\label{ccni}
We now consider the case where the regression function $\phi$ is assumed to be convex and componentwise nonincreasing. The developments here are quite similar to those in the convex case, so we omit some of the details. Given the observed values $(X_1,Y_1),\ldots,(X_n,Y_n)$, we write $\mathcal{Q}_\mathcal{X}$ for the collection of all vectors $z\in\mathbb{R}^n$ for which there is a convex, componentwise nonincreasing function $\psi$ satisfying $\psi(X_j) = z^j$ for every $j=1,\ldots,n$. We will denote by $\mathbb{R}^d_+$ and $\mathbb{R}^d_-$, respectively, the nonnegative and nonpositive orthants of $\mathbb{R}^d$. We now have the following characterizations.
\begin{lemma}\label{l9}
Let $z\in\mathbb{R}^n$. Then, $z\in\mathcal{Q}_\mathcal{X}$ if and only if the following holds for every $j=1,\ldots,n$:
\begin{equation}\nonumber
z^j = \inf\left\{ \sum_{k=1}^n \theta^k z^k : \sum_{k=1}^n \theta^k = 1,\ \vartheta + \sum_{k=1}^n \theta^k X_k = X_j,\ \theta\geq 0,\ \theta\in\mathbb{R}^n, \vartheta\in\mathbb{R}_+^d \right\}.
\end{equation}
\end{lemma}
\begin{proof} The proof is very similar to that of Lemma \ref{l2}. The difference being that we use Lemma \ref{l8} and the function
\[ h(x) = \inf\left\{ \sum_{k=1}^n \theta^k z^k : \sum_{k=1}^n \theta^k = 1,\ \vartheta + \sum_{k=1}^n \theta^k X_k = x,\ \theta\geq 0,\ \theta\in\mathbb{R}^n, \vartheta\in\mathbb{R}^d_+ \right\} \] instead of using Lemma \ref{l1} and the function $g$.
\end{proof}

The analogous dual characterization here is given in the following lemma. Its proof is just an application of the duality theorem of linear programming, so we omit it.

\begin{lemma}\label{l10}
Let $z\in\mathbb{R}^n$. Then, $z\in\mathcal{Q}_\mathcal{X}$ if and only if for every $j=1,\ldots,n$ we have
\begin{equation}\nonumber
z^j = \sup\left\{ \langle\xi,X_j\rangle + \eta : \langle\xi,X_k\rangle + \eta \leq z^k \ \forall\ k=1,\ldots,n, \ \xi \in\mathbb{R}^d_{-}, \ \eta\in\mathbb{R} \right\}.
\end{equation}
Moreover, $z\in\mathcal{Q}_\mathcal{X}$ if and only if there exist vectors $\xi_1,\ldots,\xi_n\in\mathbb{R}_-^d$ such that
\begin{equation}\nonumber
\langle \xi_j, X_k-X_j \rangle \leq z^k - z^j\ \ \forall\ k,j\in\{1,\ldots,n\}.
\end{equation}
\end{lemma}
Just as in the previous case, we can use both characterizations to show the existence and uniqueness of the vector
\[W_n = \argmin_{z\in\mathcal{Q}_\mathcal{X}}\left\{ \sum_{k=1}^n \left|Y_k - z^k\right|^2\right\}\]
and then define the nonparametric least squares estimator by
{\small \begin{equation}
\function{\hat{\varphi}_n}{x} = \inf\left\{ \sum_{k=1}^n \theta^k W_n^k : \sum_{k=1}^n \theta^k = 1, \vartheta + \sum_{k=1}^n \theta^k X_k = x, \theta \in \mathbb{R}^n_+, \vartheta\in\mathbb{R}_+^d \right\}. \nonumber
\end{equation}}
Here, the vector $W_n$ can also be computed by solving the corresponding quadratic program
\begin{equation}\nonumber
\begin{array}{rc}
\min &  \sum_{k=1}^n |Y_k - z^k|^2 \\
\textrm{subject to} & \langle \xi_k , X_j - X_k \rangle \leq z^j - z^k\ \ \forall\ k,j=1,\ldots,n \\
& \xi_1,\ldots,\xi_n\in\mathbb{R}_-^d, z\in\mathbb{R}^n.
\end{array}
\end{equation}
which differs from the program (\ref{ec7}) just because here the $\xi_j$'s have to be nonpositive. The estimator enjoys analogous finite dimensional properties to those listed in Lemma \ref{l6}. For the sake of completeness, we include them in the following lemma.

\begin{lemma}\label{l11}
Let $\hat{\varphi}_n$ be the convex, componentwise nonincreasing least squares estimator obtained from the sample $(X_1,Y_1),\ldots,(X_n,Y_n)$. Then,
\begin{enumerate}[(i)]
\item $\displaystyle \sum_{k=1}^n (\psi(x_k) - \hat{\varphi}_n(X_k))(Y_k - \hat{\varphi}_n(X_k))\leq 0$ for any convex, componentwise nonincreasing function $\psi$ which is finite on $\conv{X_1,\ldots,X_n}$;
\item $\displaystyle \sum_{k=1}^n \hat{\varphi}_n(X_k)(Y_k - \hat{\varphi}_n(X_k))= 0$;
\item $\displaystyle \sum_{k=1}^n Y_k = \sum_{k=1}^n \hat{\varphi}_n(X_k)$;
\item the set on which $\hat{\varphi}_n < \infty$ is $\conv{X_1,\ldots,X_n} + \mathbb{R}_+^d$;
\item for any $x\in\mathbb{R}^d$ the map $(X_1,\ldots,X_n,Y_1,\ldots,Y_n)\hookrightarrow\hat{\varphi}_n(x)$ is a Borel-measurable function from $\mathbb{R}^{n(d+1)}$ into $\mathbb{R}$.
\end{enumerate}
\end{lemma}
\section{Consistency of the least squares estimator}\label{sec2}
The main goal of this paper is to show that in an appropriate setting the nonparametric least squares estimator $\hat{\phi}_n$ described above is consistent for estimating the convex function $\phi$ on the set $\mathfrak{X}$. In this context, we will prove the consistency of $\hat{\phi}_n$ in both, fixed and stochastic design regression settings.

Before proceeding any further we would like to introduce some notation. For any Borel set $\mathfrak{X}\subset\mathbb{R}^d$ we will denote by $\mathcal{B}_\mathfrak{X}$ the $\sigma$-algebra of Borel subsets of $\mathfrak{X}$. Given a sequence of events $(A_n)_{n=1}^\infty$ we will be using the notation $[A_n\ \textrm{i.o.}]$ and $[A_n\ \textrm{a.a.}]$ to denote $\lsup A_n$ and $\linf A_n$, respectively.

Now, consider a convex function $f:\mathbb{R}^d\rightarrow\overline{\mathbb{R}}$. This function is said to be proper if $f(x)>-\infty$ for every $x\in\mathbb{R}^d$. The effective domain of $f$, denoted by $\textrm{Dom}(f)$, is the set of points $x\in\mathbb{R}^d$ for which $f(x)<\infty$. The subdifferential of $f$ at a point $x\in\mathbb{R}^d$ is the set $\partial f(x)\subset\mathbb{R}^d$ of all vectors $\xi$ satisfying the inequality
\[\langle \xi, h\rangle \leq f(x+h)-f(x)\ \ \ \forall\ h\in\mathbb{R}^d.\] The elements of $\partial f(x)$ are called subgradients of $f$ at $x$ (see \cite{rca}). For a set $A\subset\mathbb{R}^d$ we denote by $A^\circ$, $\overline{A}$ and $\partial A$ its interior, closure and boundary, respectively. We write $\textrm{Ext}(A)=\mathbb{R}^d\setminus\overline{A}$ for the exterior of the set $A$ and $ \textrm{diam}(A) := \sup_{x,y \in A} |x - y|$ for the diameter of $A$. We also use the sup-norm notation, i.e., for a function $g:\mathbb{R}^d\rightarrow \mathbb{R}$ we write $\|g\|_A = \sup_{x\in A}|g(x)|$.

To avoid measurability issues regarding some sets, specially those involving the random set-valued functions $\{\partial \hat{\phi}_n(x)\}_{x\in\mathfrak{X}^\circ}$, we will use the symbols $\mathbf{P}_*$ and $\mathbf{P}^*$ to denote inner and outer probabilities, respectively. We refer the reader to \cite{vw}, pages 6-15, for the basic properties of inner and outer probabilities. In this context, a sequence of (not necessarily measurable) functions $(\Psi_n)_{n=1}^\infty$ from a probability space $(\Omega,\mathcal{F},\mathbf{P})$ into $\mathbb{R}$ is said to converge to a function $\Psi$ \textit{almost surely} (see \cite{vw}, Definition 1.9.1-(iv), page 52), written $\Psi_n\cas\Psi$, if $\ip{\Psi_n\rightarrow\Psi}=1$. We will use the standard notation $\p{A}$ for the probabilities of all events $A$ whose measurability can be easily inferred from the measurability of the random variables $\{\hat{\phi}_n(x)\}_{x\in\mathfrak{X}}$, established in Lemma \ref{l6}.

Our main theorems hold for both, fixed and stochastic design schemes, and the proofs are very similar. They differ only in minor steps. Therefore, for the sake of simplicity, we will denote the observed values of the regressor variables always with the capital letters $X_n$. For any Borel set $\texttt{X}\subset\mathbb{R}^d$, we write
\[N_n(\texttt{X}) = \# \{1\le j \le n: X_j\in\texttt{X}\}.\]
The quantities $X_n$ and $N_n(\texttt{X})$ are non-random under the fixed design but random under the stochastic one.

\subsection{Fixed Design}\label{disedet}
In a ``fixed design'' regression setting we assume that the regressor values are non-random and that all the uncertainty in the model comes from the response variable. We will now list a set of assumptions for this type of design. The one-dimensional case has been proven, under different regularity conditions, in \cite{hp76}.
\begin{enumerate}
\item[(A1)] We assume that we have a sequence $\left(X_n,Y_n\right)_{n=1}^\infty$ satisfying
\[ Y_k=\phi(X_k) + \epsilon_k\]
where $(\epsilon_n)_{n=1}^\infty$ is an i.i.d. sequence with $\e{\epsilon_j} = 0$, $\e{\epsilon_j^2} = \sigma^2<\infty$ and $\phi:\mathbb{R}^d\rightarrow\mathbb{R}$ is a proper convex function.
\item[(A2)] The non-random sequence $(X_n)_{n=1}^\infty$ is contained in a closed, convex set $\mathfrak{X}\subset\mathbb{R}^d$ with $\mathfrak{X}^\circ \neq \emptyset$ and $\mathfrak{X}\subset \textrm{Dom}(\phi)$.
\item[(A3)] We assume the existence of a Borel measure $\nu$ on $\mathfrak{X}$ satisfying:
    \begin{enumerate}[(i)]
    \item $\{\texttt{X}\in\mathcal{B}_\mathfrak{X} : \nu(\texttt{X})=0\} = \{\texttt{X}\in\mathcal{B}_\mathfrak{X} : \texttt{X} \textrm{ has Lebesgue measure } 0\}$.
    \item $\frac{1}{n}N_n (\texttt{X})\rightarrow \nu(\texttt{X})$ for any open rectangle $\texttt{X}\subset\mathfrak{X}^\circ$.
    \end{enumerate}
\end{enumerate}
Condition (A1) may be replaced by the following:
\begin{enumerate}
\item[(A4)] We assume that we have a sequence $\left(X_n,Y_n\right)_{n=1}^\infty$ satisfying \[ Y_k=\phi(X_k) + \epsilon_k\] where $\phi:\mathbb{R}^d\rightarrow\mathbb{R}$ is a proper convex function and $(\epsilon_n)_{n=1}^\infty$ is an independent sequence of random variables satisfying
\begin{enumerate}[(i)]
\item $\e{\epsilon_n} = 0$ $\forall$ $n\in\mathbb{N}$ and $ \linf  \frac{1}{n}\sum_{k=1}^n \e{|\epsilon_k|} > 0$.
\item $ \sum_{n=1}^\infty \frac{\var{\epsilon_n^2}}{n^2} <\infty$.
\item $ \sup_{n\in\mathbb{N}}\{\e{\epsilon_n^2}\} <\infty$.
\end{enumerate}
Under these conditions we define $\sigma^2 := \lsup_{n\rightarrow\infty}\frac{1}{n}\sum_{j=1}^n \e{\epsilon_j^2}$.
\end{enumerate}
The raison d'etre of condition (A4) is to allow the variance of the error terms to depend on the regressors. We make the distinction between (A1) and (A4) because in the case of i.i.d. errors it is enough to require a finite second moment to ensure consistency.

\subsection{Stochastic Design}\label{disestoc}
In this setting we assume that $(X_n,Y_n)_{n=1}^\infty$ is an i.i.d. sequence from some Borel probability measure $\mu$ on $\mathbb{R}^{d+1}$. Here we make the following assumptions on the measure $\mu$:
\begin{enumerate}[(I)]
\item[(A5)] There is a closed, convex set $\mathfrak{X}\subset\mathbb{R}^d$ with $\mathfrak{X}^\circ \neq \emptyset$ such that $\mu (\mathfrak{X}\times\mathbb{R}) =1$. Also,
    \[ \int_{\mathfrak{X}\times\mathbb{R}} y^2 \mu(dx,dy) <\infty. \]
\item[(A6)] There is a proper convex function $\phi:\mathbb{R}^d\rightarrow\mathbb{R}$ with $\mathfrak{X}\subset\textrm{Dom}(\phi)$ such that whenever $(X,Y)\sim\mu$ we have $\e{Y-\phi(X)|X} = 0$ and $\e{|Y - \phi(X)|^2} = \sigma^2 <\infty$. Thus, $\phi$ is the regression function.
\item[(A7)] Denoting by $\nu(\cdot) = \mu((\cdot)\times\mathbb{R})$ the $x$-marginal of $\mu$, we assume that \[\{\texttt{X}\in\mathcal{B}_\mathfrak{X} : \nu(\texttt{X})=0\} = \{\texttt{X}\in\mathcal{B}_\mathfrak{X} : \texttt{X} \textrm{ has Lebesgue measure } 0\}.\]
\end{enumerate}

We wish to point out some conclusions that one can draw from these assumptions. Consider the class of functions \[\mathcal{K}_\mu := \left\{ \psi:\mathbb{R}^d \rightarrow \mathbb{R}\  | \ \psi \textrm{ is  convex with } \int |\psi(x)|^2 \nu(dx) < \infty \right\}. \]
Then for any $\texttt{X}\subset\mathfrak{X}$ the following holds
\[ \int_{\texttt{X}\times\mathbb{R}} \psi(x)(y-\phi(x))\mu(dx,dy) = 0 \ \ \forall \psi \ \in \mathcal{K}_\mu;\]
so we get that $\phi$ is in fact the element of $\mathcal{K}_\mu$ which is the closest to $Y$ in the Hilbert space $\mathbb{L}^2(\texttt{X} \times\mathbb{R}, \mathcal{B}_{\texttt{X} \times\mathbb{R}},\mu)$. This follows from Moreau's decomposition theorem (see the proof of Lemma \ref{l6}).

Additionally, conditions \{A5-A7\} allow for stochastic dependency between the error variable $Y-\phi(X)$ and the regressor $X$. Although some level of dependency can be put to satisfy conditions \{A2-A4\}, the measure $\mu$ allows us to take into account some cases which wouldn't fit in the fixed design setting (even by conditioning on the regressors).

\subsection{Main results}\label{principales}
We can now state the two main results of this paper. The first result shows that assuming only the convexity of $\phi$, the least squares estimator can be used to consistently estimate both $\phi$ and its subdifferentials $\partial \phi(x)$.
\begin{thm}\label{teo1}
Under any of \{A1-A3\}, \{A2-A4\} or \{A5-A7\} we have,
\begin{enumerate}[(i)]
\item $\displaystyle \p{\sup_{x\in\texttt{X}}\{|\hat{\phi}_n(x) - \phi(x)|\} \rightarrow 0\ \textrm{ for any compact set } \texttt{X}\subset\mathfrak{X}^\circ} = 1. $
\item For every $x\in\mathfrak{X}^\circ$ and every $\xi\in\mathbb{R}^d$
    \[\displaystyle \lsup_{n\rightarrow \infty} \lim_{h \downarrow 0} \frac{\hat{\phi}_n(x + h\xi) - \hat{\phi}_n(x)}{h} \leq \lim_{h \downarrow 0} \frac{\phi(x + h\xi) - \phi(x)}{h}\ \textrm{almost surely}.\]
\item Denoting by \textbf{B} the unit ball (w.r.t. the Euclidian norm) we have
\[ \ip{\partial \hat{\phi}_n (x) \subset \partial \phi (x) + \epsilon\textbf{B}\ \ \emph{a.a.}} =1\ \ \ \forall\ \epsilon>0, \ \forall\ x\in\mathfrak{X}^\circ.\]
\item If $\phi$ is differentiable at $x\in\mathfrak{X}^\circ$, then
\[ \sup_{\xi\in\partial \hat{\phi}_n (x)}\{ |\xi - \nabla \phi(x)|\}\cas 0. \]
\end{enumerate}
\end{thm}
Our second result states that assuming differentiability of $\phi$ on the entire $\mathfrak{X}^\circ$ allows us to use the subdifferentials of the least squares estimator to consistently estimate $\nabla \phi$ uniformly on compact subsets of $\mathfrak{X}^\circ$.
\begin{thm}\label{teo2}
If $\phi$ is differentiable on $\mathfrak{X}^\circ$, then under any of \{A1-A3\}, \{A2-A4\} or \{A5-A7\} we have,
\[ \ip{\sup_{\begin{subarray}{c} x\in \texttt{X}\\ \xi\in\partial \hat{\phi}_n (x)\end{subarray}}\{ |\xi - \nabla \phi(x)|\}\rightarrow 0 \ \textrm{ for any compact set } \texttt{X}\subset\mathfrak{X}^\circ} = 1.\]
\end{thm}
\subsection{Proof of the main results}
Before embarking on the proofs, one must notice that there are some statements which hold true under any of \{A1-A3\}, \{A2-A4\} or \{A5-A7\}. We list the most important ones below, since they'll be used later.
\begin{itemize}
\item For any set $\texttt{X}\subset \mathfrak{X}$ we have
\begin{equation}
 \frac{N_n(\texttt{X})}{n}\cas \nu(\texttt{X}). \label{ec41}
\end{equation}
\item The strong law of large numbers implies that for any Borel set $\texttt{X}\subset\mathfrak{X}$ with positive Lebesgue measure we have
    \begin{eqnarray}
    \frac{1}{N_n(\texttt{X})}\sum_{\begin{subarray}{c} X_k\in\texttt{X}\\1\leq k \leq n\end{subarray}} (Y_k-\phi(X_k)) \cas 0\label{ec42}
    \end{eqnarray}
    and also
    \begin{eqnarray}
    \lsup_{n\rightarrow\infty} \frac{1}{n}\sum_{1\leq k \leq n} (Y_k-\phi(X_k))^2 =  \sigma^2 \textrm{ a.s.}\label{ec43}
    \end{eqnarray}
    We would like to point out that in the case of condition A4, A4-(iii) allows us to obtain (\ref{ec42}) from an application of a version of the strong law of large number for uncorrelated random variables, as it appears in \cite{chung}, page 108, Theorem 5.1.2. Similarly, condition A4-(ii) implies that we can apply a version the strong law of large numbers for independent random variables as in \cite{pwm}, Lemma 12.8, page 118 or in \cite{fol}, Theorem 10.12, page 322 to obtain (\ref{ec43}). \newline

\item For any Borel subset $\texttt{X}\subset\mathfrak{X}$ with positive Lebesgue measure,
    \begin{equation}
    \#\{n\in\mathbb{N}: X_n\in\texttt{X}\}\cas +\infty\label{ec44}
    \end{equation}
\end{itemize}

{\it Proof of Theorem \ref{teo1}.} We will only make distinctions among the design schemes in the proof if we are using any property besides (\ref{ec41}), (\ref{ec42}), (\ref{ec43}) or (\ref{ec44}). For the sake of clarity, we divide the proof in steps. \\

\noindent {\bf Step I:} We start by showing that for any set with positive Lebesgue measure there is a uniform band around the regression function (over that set) such that $\hat{\phi}_n$ comes within the band at least at one point for all but finitely many $n$'s. This fact is stated in the following lemma (proved in Section \ref{pl15}).
\begin{lemma}\label{l15}
For any set $\texttt{X}\subset\mathfrak{X}$ with positive Lebesgue measure we have,
\[\p{\inf_{x\in\texttt{X}}\left\{ |\hat{\phi}_n (x) - \phi(x)|\right\} \geq M\ \emph{i.o.}} = 0\ \ \forall\ M>\frac{\sigma}{\sqrt{\nu(\texttt{X})}}.\]
\end{lemma}
\noindent {\bf Step II:} The idea is now to use the convexity of both, $\phi$ and $\hat{\phi}_n$, to show that the previous result in fact implies that the sup-norm of $\hat{\phi}_n$ is uniformly bounded on compact subsets of $\mathfrak{X}^\circ$. We achieve this goal in the following two lemmas (whose proofs are given in Sections \ref{pl16} and \ref{pl17} respectively).
\begin{lemma}\label{l16}
Let $\texttt{X}\subset\mathfrak{X}^\circ$ be compact with positive Lebesgue measure. Then, there is a positive real number $K_\texttt{X}$ such that
\[ \p{ \inf_{x\in\texttt{X}} \{\hat{\phi}_n (x)\} < - K_\texttt{X} \emph{ i.o.}} = 0. \]
\end{lemma}
\begin{lemma}\label{l17}
Let $\texttt{X}\subset\mathfrak{X}^\circ$ be a compact set with positive Lebesgue measure. Then, there is $K_\texttt{X} > 0$
such that
\[ \p{ \sup_{x\in\texttt{X}} \{\hat{\phi}_n (x)\} \geq K_\texttt{X}\ \emph{ i.o.}} = 0. \]
\end{lemma}
\noindent {\bf Step III:} Convex functions are determined by their subdifferential mappings (see \cite{rca}, Theorem 24.9, page 239). Moreover, having a uniform upper bound $K_\texttt{X}$ for the norms of all the subgradients over a compact region $\texttt{X}$ imposes a Lipschitz continuity condition on the convex function over $\texttt{X}$ (see \cite{rca}, Theorem 24.7, page 237); the Lipschitz constant being $K_\texttt{X}$. For these reasons, it is important to have a uniform upper bound on the norms of the subgradients of $\hat{\phi}_n$ on compact regions. The following lemma (proved in Section \ref{pl18}) states that this can be achieved.
\begin{lemma}\label{l18}
Let $\texttt{X}\subset\mathfrak{X}^\circ$ be a compact set with positive Lebesgue measure. Then, there is $K_\texttt{X}>0$ such that
\[ \op{\sup_{\begin{subarray}{c}\xi\in\partial \hat{\phi}_n (x) \\ x\in\texttt{X}\end{subarray}}\{|\xi|\} > K_\texttt{X} \emph{ i.o.}} = 0. \]
\end{lemma}
\noindent {\bf Step IV:} For the next results we need to introduce some further notation. We will denote by $\mu_n$ the empirical measure defined on $\mathbb{R}^{d+1}$ by the sample $(X_1,Y_1),\ldots,(X_n,Y_n)$. In agreement with \cite{vw}, given a class of functions $\mathcal{G}$ on $D\subset\mathbb{R}^{d+1}$, a seminorm $\left\|\cdot\right\|$ on some space containing $\mathcal{G}$ and $\epsilon>0$ we denote by $N(\epsilon,\mathcal{G},\|\cdot\|)$ the $\epsilon$ covering number of $\mathcal{G}$ with respect to $\|\cdot\|$.

Although Lemmas \ref{l19} and \ref{l20} may seem unrelated to what has been done so far, they are crucial for the further developments. Lemma 3.5 (proved in Section \ref{pl19}) shows that the class of convex functions is not very complex in terms of entropy. Lemma \ref{l20} is a uniform version of the strong law of large numbers which proves vital in the proof of Lemma \ref{l21}.
\begin{lemma}\label{l19}
Let $\texttt{X}\subset\mathfrak{X}^\circ$ be a compact rectangle with positive Lebesgue measure. For $K > 0$ consider the class $\mathcal{G}_{K,\texttt{X}}$ of all functions of the form $\psi(X)(Y-\phi(X))\ind{\texttt{X}}(X)$ where $\psi$ ranges over the class $\mathcal{D}_{K,\texttt{X}}$ of all proper convex functions which satisfy
\begin{enumerate}[(a)]
\item $\|\psi\|_\texttt{X} \leq K$;
\item $\displaystyle \bigcup_{\begin{subarray}{c} \xi\in\partial \psi(x)\\ x\in\texttt{X}\end{subarray}} \left\{ \xi\right\} \subset [-K,K]^d$.
\end{enumerate}
Then, for any $\epsilon > 0$ we have \[ \lsup_{n\rightarrow\infty} N(\epsilon, \mathcal{G}_{K,\texttt{X}}, \mathbb{L}_1 (\texttt{X}\times\mathbb{R},\mu_n)) <\infty\ \ \textrm{almost surely},\]
and there is a positive constant $A_\epsilon < \infty$, depending only on $(X_1,\ldots,X_n)$, $K$ and $\texttt{X}$, such that the covering numbers $N (\frac{\epsilon}{n}\sum_{j=1}^n|Y_j - \phi(X_j)|, \mathcal{G}_{K,\texttt{X}}, \mathbb{L}_1 (\texttt{X}\times\mathbb{R},\mu_n) )$ are bounded above by $A_\epsilon$, for all $n\in\mathbb{N}$, almost surely.
\end{lemma}
The proofs of Lemmas \ref{l20} and \ref{l21} (given in Sections \ref{pl20} and \ref{pl21} respectively) are the only parts in the whole proof where we must treat the different design schemes separately. To make the argument work, a small lemma (proved in Section \ref{plfdp}) for the set of conditions \{A2-A4\} is required. We include it here for the sake of completeness and to point out the difference between the schemes.
\begin{lemma}\label{lfdp}
Consider the set of conditions \{A2-A4\} and a subsequence $(n_k)_{k=1}^\infty$ such that
\[ \lim_{k\rightarrow\infty} \frac{1}{n_k}\sum_{j=1}^{n_k} \e{\epsilon_j^2} = \sigma^2.\]
Let $(\texttt{X}_m)_{m=1}^\infty$ be a an increasing sequence of compact subsets of $\mathfrak{X}$ satisfying $\nu(X_m)\rightarrow 1$. Then,
\[ \lim_{m\rightarrow\infty}\linf_{k\rightarrow\infty} \frac{1}{n_k}\sum_{\{1\leq j\leq n_k: X_j\in\texttt{X}_m\}}\e{\epsilon_j^2} = \sigma^2.\]
\end{lemma}
We are now ready to state the key result on the uniform law of large numbers.
\begin{lemma}\label{l20}
Consider the notation of Lemma \ref{l19} and let $\texttt{X}\subset\mathfrak{X}^\circ$ be any finite union of compact rectangles with positive Lebesgue measure. Then,
\[ \sup_{\psi\in\mathcal{D}_{K,\texttt{X}}} \left\{ \left|\frac{1}{n}\sum_{\{1\leq j\leq n: X_j\in\texttt{X}\}} \psi(X_j)(Y_j - \phi(X_j)) \right|\right\}\cas 0.\]
\end{lemma}
\noindent {\bf Step V:} With the aid of all the results proved up to this point, it is now possible to show that Lemma \ref{l15} is in fact true if we replace $M$ by an arbitrarily small $\eta > 0$. The proof of the following lemma is given in Section \ref{pl21}.
\begin{lemma}\label{l21}
Let $\texttt{X}\subset\mathfrak{X}^\circ$ be any compact set with positive Lebesgue measure. Then,
\begin{enumerate}[(i)]
\item $\displaystyle \p{\inf_{x\in\texttt{X}}\{  \phi(x) - \hat{\phi}_n (x) \} \geq \eta\ \emph{ i.o.}} = 0\ \ \ \forall\ \eta > 0,$
\item $\displaystyle \p{\sup_{x\in\texttt{X}}\{  \phi(x) - \hat{\phi}_n (x) \} \leq -\eta\ \emph{ i.o.}} = 0\ \ \ \forall\ \eta > 0.$
\end{enumerate}
\end{lemma}
\noindent {\bf Step VI:} Combining the last lemma with the fact that we have a uniform bound on the norms of the subgradients on compacts, we can state and prove the consistency result on compacts. This is done in the next lemma (proof included in Section \ref{pl23}).
\begin{lemma}\label{l23}
Let $\texttt{X}\subset\mathfrak{X}^\circ$ be a compact set with positive Lebesgue measure. Then,
\begin{enumerate}[(i)]
\item $\displaystyle \p{\inf_{x\in\texttt{X}}\{ \hat{\phi}_n (x) - \phi(x) \} < -\eta\ \emph{ i.o.}} = 0\ \ \ \forall\ \eta > 0,$
\item $\displaystyle \p{\sup_{x\in\texttt{X}}\{ \hat{\phi}_n (x) - \phi(x) \} > \eta\ \emph{ i.o.}} = 0\ \ \ \forall\ \eta > 0,$
\item $\displaystyle \sup_{x\in\texttt{X}} \{ |\hat{\phi}_n (x) - \phi(x)|\} \cas 0$.
\end{enumerate}
\end{lemma}
\noindent {\bf Step VII:} We can now complete the proof of Theorem \ref{teo1}. Consider the class $\mathfrak{C}$ of all open rectangles $\mathcal{R}$ such that $\overline{\mathcal{R}}\subset\mathfrak{X}^\circ$ and whose vertices have rational coordinates. Then, $\mathfrak{C}$ is countable and $\bigcup_{\mathcal{R}\in\mathfrak{C}}\mathcal{R} = \mathfrak{X}^\circ$. Observe that Lemmas \ref{l16} and \ref{l17} imply that for any finite union $A := \mathcal{R}_1\cup\cdots\cup\mathcal{R}_m$ of open rectangles $\mathcal{R}_1,\ldots,\mathcal{R}_m\in\mathfrak{C}$ there is, with probability one, $n_0\in\mathbb{N}$ such that the sequence $(\hat{\phi}_n)_{n=n_0}^\infty$ is finite on $\conv{A}$. From Lemma \ref{l23} we know that the least squares estimator converges at all rational points in $\mathfrak{X}^\circ$ with probability one. Then, Theorem 10.8, page 90 of \cite{rca} implies that ($i$) holds if $\mathfrak{X}^\circ$ is replaced by the convex hull of a finite union of rectangles belonging to $\mathfrak{C}$. Since there are countably many of such unions and any compact subset of $\mathfrak{X}^\circ$ is contained in one of those unions, we see that $(i)$ holds. An application of Theorem 24.5, page 233 of \cite{rca} on an open rectangle $C$ containing $x$ and satisfying $\overline{C}\subset\mathfrak{X}^\circ$ gives ($ii$) and ($iii$). Note that ($iv$) is a consequence of ($iii$). $\hfill \square \\$

{\it Proof of Theorem \ref{teo2}.} To prove the desired result we need the following lemma (whose proof is provided in Section \ref{pl22}) from convex analysis. The result is an extension of Theorem 25.7, page 248 of \cite{rca}, and might be of independent interest.

\begin{lemma}\label{l22}
Let $\mathcal{C}\subset\mathbb{R}^d$ be an open, convex set and $f$ a convex function which is finite and differentiable on $\mathcal{C}$. Consider a sequence of convex functions $(f_n)_{n=1}^\infty$ which are finite on $\mathcal{C}$ and such that $f_n\rightarrow f$ pointwise on $\mathcal{C}$. Then, if $\texttt{X}\subset\mathcal{C}$ is any compact set,
\[ \sup_{\begin{subarray}{c}x\in\texttt{X}\\ \xi\in\partial f_n (x)\end{subarray}}\left\{ |\xi - \nabla f(x)|\right\} \rightarrow 0. \]
\end{lemma}
Defining the class $\mathfrak{C}$ of open rectangles as in the proof of Theorem \ref{teo1}, one can use a similar argument to obtain Theorem \ref{teo2} from an application of Theorem \ref{teo1} and the previous lemma. $\hfill \square$

\subsection{The componentwise nonincreasing case}
The regression function $\phi$ is now assumed to be convex and componentwise nonincreasing. Recalling the notation defined in Section \ref{ccni}, we now have that Theorems \ref{teo1} and \ref{teo2} still hold with $\hat{\phi}_n$ replaced by $\hat{\varphi}_n$. In view of the fact that the proof of the results is very similar to that when $\phi$ is just convex, we omit the proof and sketch the main differences. The proof of the main results in Section \ref{sec2} relied essentially on two key facts:
\begin{enumerate}[(i)]
\item The finite sample properties of $\hat{\phi}_n$ established in Lemma \ref{l6}.

\item The vector $(\hat{\phi}_n(X_1), \ldots, \hat{\phi}_n(X_n))' \in \mathbb{R}^n$ is the $\mathcal{L}_2$ projection of $(Y_1,\ldots,$ $Y_n)$ on the closed, convex cone $\mathcal{K}_\mathcal{X}$ of all evaluations of proper convex functions on $(X_1,\ldots,$ $ X_n)$. Also, note that $(\phi(X_1),\ldots,\phi(X_n))'\in\mathcal{K}_\mathcal{X}.$
\end{enumerate}
We know from Lemma \ref{l11} that $\hat{\varphi}_n$ has similar finite sample properties as its convex counterpart. Note that if $\phi$ is convex and componentwise nonincreasing $(\phi(X_1),\ldots, $ $\phi(X_n))' \in \mathcal{Q}_\mathcal{X}$ and $(\hat{\varphi}_n(X_1), \ldots, \hat{\varphi}_n(X_n))' \in \mathbb{R}^n$ is the $\mathcal{L}_2$ projection of $(Y_1,\ldots,Y_n)$ onto $\mathcal{Q}_\mathcal{X}$.

From these considerations and the nature of the arguments used to prove Theorems \ref{teo1} and \ref{teo2}, it follows that all but one of those arguments carry forward to the componentwise nonincreasing case; the only difference being the entropy calculation of Lemma \ref{l19}. At some point in that proof, one breaks the rectangle $[-K,K]^d$ into a family of subrectangles in order to approximate the subdifferentials of the class $\mathcal{D}_{K,\texttt{X}}$. It is easily seen that the same argument holds in the componentwise nonincreasing case if one instead uses a partition of $[-K,0]^d$ to approach the subdifferentials of the corresponding class $\mathcal{D}_{K,\texttt{X}}$ for componentwise nonincreasing convex functions. By doing this, the resulting function $g$ will be convex and componentwise nonincreasing and (\ref{ec30}), (\ref{ec31}) and (\ref{ec32}) will still hold for the corresponding class $\mathcal{H}_{n,\epsilon}$. Then, the conclusions of Lemma \ref{l19} are also true for the componentwise nonincreasing case and we can conclude that our main results are valid in this case too.

\section{Proofs of the lemmas}\label{sec4}
Here we prove the lemmas involved in the proof of the main theorem. To prove these, we will need additional auxiliary results from matrix algebra and convex analysis, which may be of independent interest and are proved in the Appendix.

\subsection{Proof of Lemma \ref{l15}}\label{pl15}
We will first show that the event \\ $\left[\inf_{x\in\texttt{X}}\left\{ \hat{\phi}_n (x) - \phi(x)\right\} \geq M\ \textrm{i.o.}\right]$ has probability zero. Under this event, there is a subsequence $(n_k)_{k=1}^\infty$ such that $\inf_{x\in\texttt{X}}\left\{ \hat{\phi}_{n_k} (x) - \phi(x)\right\} \geq M$ $\forall$ $k\in\mathbb{N}$. Then (\ref{ec42}) implies that for this subsequence, with probability one, we have
\begin{eqnarray}
\lsup_{k\rightarrow\infty} \frac{1}{N_{n_k}(\texttt{X})}\sum_{X_j\in\texttt{X}} \{Y_j - \hat{\phi}_{n_k} (X_j)\} &\leq& -M.\label{ec20}
\end{eqnarray}
On the other hand, it is seen (by solving the corresponding quadratic programming problems; see, e.g., Exercise 16.2, page 484 of \cite{nowr}) that for any $\eta > 0$, $m\in\mathbb{N}$
\begin{eqnarray}
\inf \left\{ \frac{1}{m}\sum_{1\leq j\leq m} |\xi^j|^2 \ : \frac{1}{m}\sum_{1\leq j\leq m} \xi^j \geq \eta,\ \xi\in\mathbb{R}^{m} \right\} &=& \eta^2, \label{eq:Inf1} \\
\inf \left\{ \frac{1}{m}\sum_{1\leq j\leq m} |\xi^j|^2 \ : \frac{1}{m}\sum_{1\leq j\leq m} \xi^j \leq -\eta,\ \xi\in\mathbb{R}^{m} \right\} &=& \eta^2. \label{eq:Inf2}
\end{eqnarray}
For $0 < \delta < M$, using (\ref{eq:Inf2}) with $\eta = M - \delta$ together with (\ref{ec44}) and (\ref{ec20}) we get that, with probability one, we must have
\[ \linf_{k\rightarrow\infty} \frac{1}{n_{k}} \sum_{j=1}^{n_k} (Y_j - \hat{\phi}_{n_k}(X_j))^2 \geq \nu (\texttt{X}) (M-\delta)^2. \]
Letting $\delta\rightarrow 0$ we actually get
\[ \linf_{k\rightarrow\infty} \frac{1}{n_{k}} \sum_{j=1}^{n_k} (Y_j - \hat{\phi}_{n_k}(X_j))^2  \geq \nu(\texttt{X})M^2 > \sigma^2 = \lsup_{k\rightarrow\infty} \frac{1}{n_{k}} \sum_{j=1}^{n_k} (Y_j - \phi(X_j))^2\ \ \textrm{a.s.} \]
which is impossible because $\hat{\phi}_{n_k}$ is the least squares estimator. Therefore,
\[\p{\inf_{x\in\texttt{X}}\left\{ \hat{\phi}_n (x) - \phi(x)\right\} \geq M\ \textrm{i.o.}} = 0.\]
A similar argument now using (\ref{eq:Inf1}) gives
\[\p{\sup_{x\in\texttt{X}}\left\{ \hat{\phi}_n (x) - \phi(x)\right\} \leq -M\ \textrm{i.o.}} = 0, \]
which completes the proof of the lemma. $\hfill \square \\$

Before we prove Lemmas \ref{l16} and \ref{l17}, we need some additional results from matrix algebra. For convenience, we state them here, but postpone their proofs to Section \ref{apapendice} in the Appendix.

We first introduce some notation. We write $\textbf{e}_j \in \mathbb{R}^d$ for the vector whose components are given by $\mathbf{e}_j^k = \delta_{jk}$, where $\delta_{jk}$ is the Kronecker $\delta$. We also write $\textbf{e} = \textbf{e}_1 + \ldots + \textbf{e}_d$ for the vector of ones in $\mathbb{R}^d$. For $\alpha\in\{-1,1\}^d$ we write
\[\mathcal{R}_\alpha = \left\{\sum_{k=1}^d \theta^k \alpha^k \textbf{e}_k : \theta\geq 0, \theta\in\mathbb{R}^d\right\}\] for the orthant in the $\alpha$ direction.
For any hyperplane $\mathcal{H}$ defined by the normal vector $\xi\in\mathbb{R}^d$ and the intercept $b\in\mathbb{R}$, we write $\mathcal{H} = \{ x\in\mathbb{R}^d : \langle\xi,x\rangle = b\}$, $\mathcal{H}^+ = \{ x\in\mathbb{R}^d : \langle\xi,x\rangle > b\}$ and $\mathcal{H}^- = \{ x\in\mathbb{R}^d : \langle\xi,x\rangle < b\}$. For $r>0$ and $x_0\in\mathbb{R}^d$ we will write $B(x_0,r) = \{x\in\mathbb{R}^d: |x-x_0|<r\}$. We denote by $\mathbb{R}^{d\times d}$ the space of $d \times d$ matrices endowed with the topology defined by the $\|\cdot\|_2$ norm (where $ \| A\|_2 = \sup_{|x|\leq 1} \{|Ax|\}$ and can be shown to be equal to the largest singular value of $A$; see \cite{masp}).

\begin{lemma}\label{l12}
Let $r>0$. There is a constant $R_r>0$, depending only on $r$ and $d$, such that for any $\rho_*\in (0,R_r)$ there are $\rho,\rho^* > 0$ with the property: for any $\alpha\in\{-1,1\}^d$ and any $d$-tuple of vectors $\beta=\{x_1,\ldots,x_d\}\subset\mathbb{R}^d$ such that $x_j \in B(\alpha^j r\textbf{e}_j, \rho)$ $\forall$ $j=1,\ldots,d$, there is a unique pair $(\xi_{\alpha,\beta},b_{\alpha,\beta})$, with $\xi_{\alpha,\beta}\in\mathbb{R}^d$, $|\xi_{\alpha,\beta}| =1$ and  $b_{\alpha,\beta}>0$ for which the following statements hold:
\begin{enumerate}[(i)]
\item $\beta$ form a basis for $\mathbb{R}^d$.
\item $x_1,\ldots,x_d\in\mathcal{H}_{\alpha,\beta} : = \{x\in\mathbb{R}^d:\langle \xi_{\alpha,\beta}, x \rangle = b_{\alpha,\beta} \}$.
\item $\displaystyle \min_{1\leq j \leq d} \{ |\xi_{\alpha,\beta}^j |\} > 0$.
\item $B(0,\rho_*)\subset\mathcal{H}_{\alpha,\beta}^-$.
\item $\{x\in\mathbb{R}^d: |x| \geq \rho^*\}\cap\mathcal{R}_\alpha \subset \mathcal{H}_{\alpha,\beta}^+$.
\item $B(-\alpha^j r \textbf{e}_j,\rho)\subset \{x\in\mathbb{R}^d:\langle \xi_{\alpha,\beta}, x \rangle < 0\}$ for all $j=1,\ldots,d$.
\item For any $w_1\in B\left(0,\frac{\rho_*}{16\sqrt{d}}\right)$ and $w_2\in B\left(\frac{3\rho_*}{8\sqrt{d}}\alpha,\frac{\rho_*}{8\sqrt{d}}\right)$ we have
    \[ \min_{1\leq j\leq d}\left\{ \left(X_\beta^{-1}\left(w_1 + t(w_2-w_1)\right)\right)^j \right\}>0\ \ \forall \ t\geq 1 \]
    where $X_\beta = (x_1,\ldots,x_d)\in\mathbb{R}^{d\times d}$ is the matrix whose $j$'th column is $x_j$.
\end{enumerate}
\end{lemma}
Figure \ref{figura1}a illustrates the above lemma when $d=2$ and $\alpha = (1,1)$. The lemma states that whatever points $x_1$ and $x_2$ are taken inside the circles of radius $\rho$ around $\alpha^1 r\textbf{e}_1$ and $\alpha^2 r\textbf{e}_2$, respectively, $B(0,\rho_*)$ and $\{x\in\mathbb{R}^d: |x| \ge \rho^*\}\cap\mathcal{R}_\alpha$ are contained, respectively, in the half-spaces $\mathcal{H}_{\alpha,\beta}^-$ and $\mathcal{H}_{\alpha,\beta}^+$. Assertion $(vii)$ of the lemma implies that all the points in the half line $\{w_1 + t(w_2-w_1\}_{t \ge 1}$ should have positive co-ordinates with respect to the basis $\beta$ as they do with respect to the basis $\{\alpha^j \textbf{e}_j\}_{j=1}^d$. We refer the reader to Section \ref{pl12} for a complete proof of Lemma \ref{l12}.

We now state two other useful results, namely Lemma \ref{l13} and Lemma \ref{l14}, but defer their proofs to Section \ref{pl13} and Section \ref{pl14} respectively.
\begin{figure}
\centering \includegraphics[height=6cm,width=6cm]{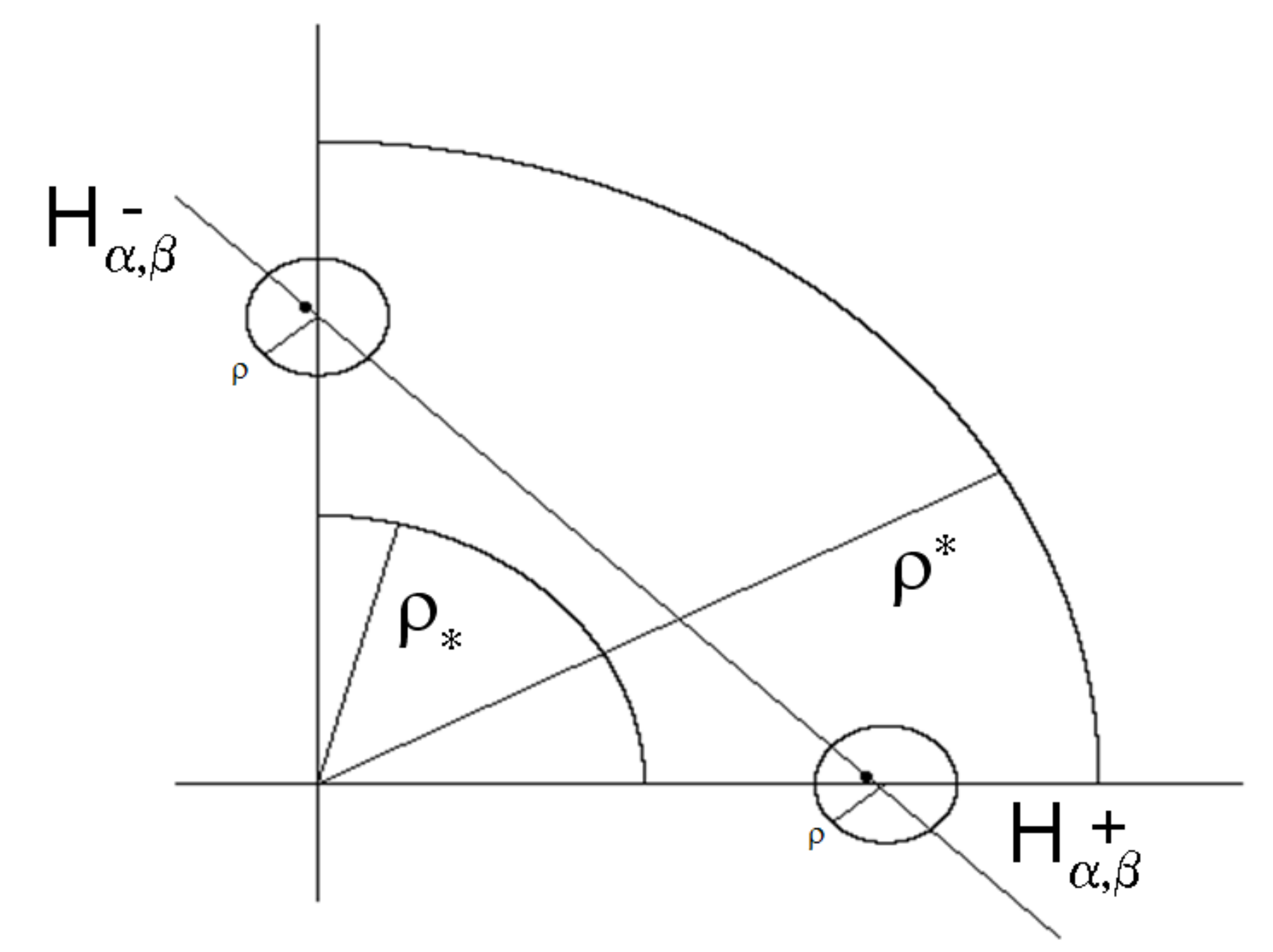}\includegraphics[height=6cm,width=7cm]{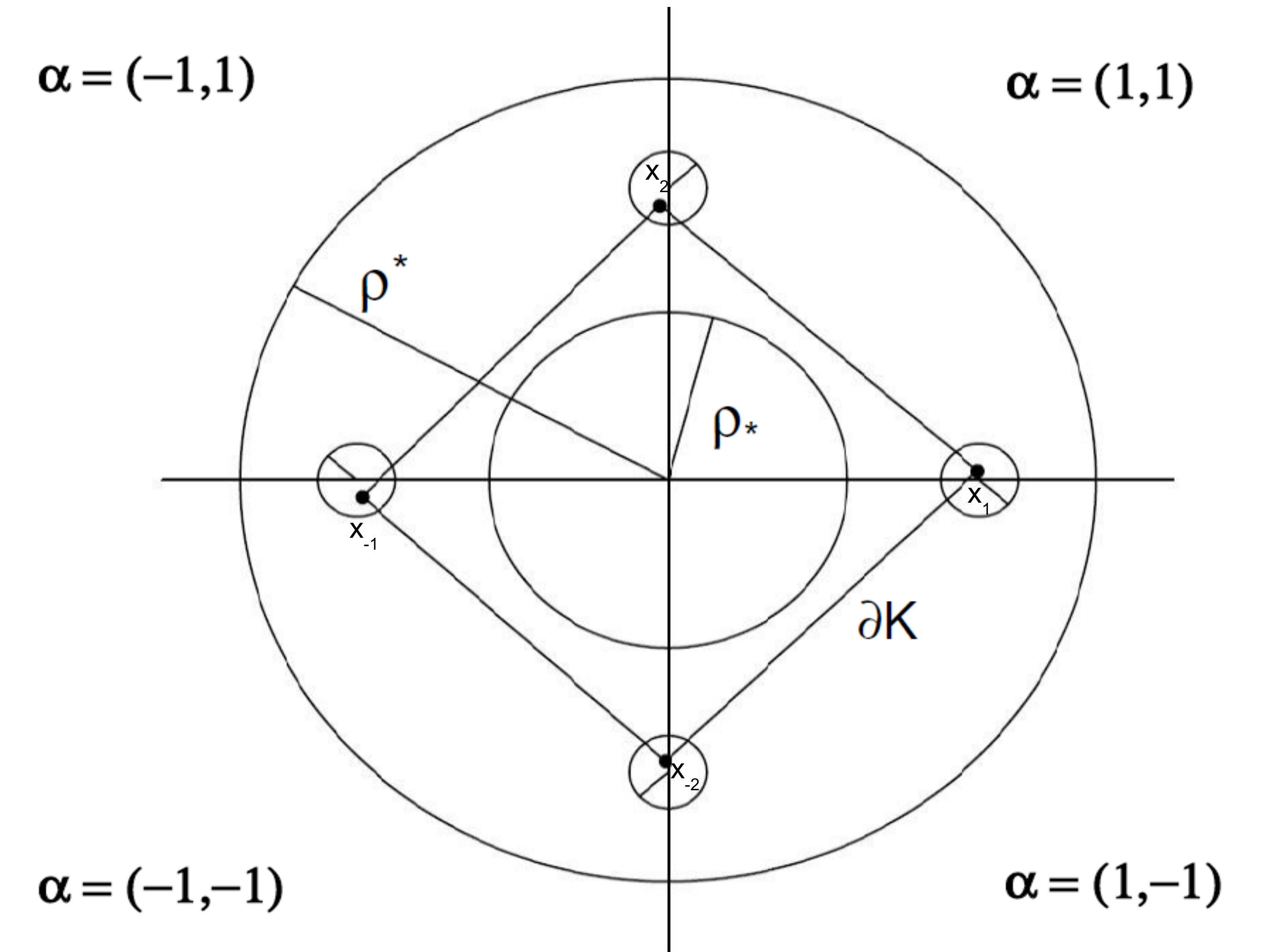}
\caption{Explanatory diagram for (a) Lemma \ref{l12} (left panel); (b) Lemma \ref{l13} (right panel).}\label{figura1}
\end{figure}
\begin{lemma}\label{l13}
Let $r>0$ and consider the notation of Lemma \ref{l12} with the positive numbers $\rho$, $\rho_*$ and $\rho^*$ as defined there. Take 2d vectors $\{x_{\pm 1},\ldots, x_{\pm d}\}$ $\subset\mathbb{R}^d$ such that $x_{\pm j} \in B(\pm r \textbf{e}_j, \rho)$ and for $\alpha\in\{-1,1\}^d$ write $\beta_\alpha = \{x_{\alpha^1 1},x_{\alpha^2 2},$ $ \ldots, x_{\alpha^d d}\}$, $\xi_\alpha = \xi_{\alpha,\beta_\alpha}$, $b_\alpha = b_{\alpha,\beta_\alpha}$ and $\mathcal{H}_\alpha = \mathcal{H}_{\alpha,\beta}$, all in agreement with the setting of Lemma \ref{l12}. Then, if $K = \conv{x_{\pm 1},\ldots,x_{\pm d}}$ we have:
\begin{enumerate}[(i)]
\item $K = \bigcap_{\alpha\in\{-1,1\}^d} \{x\in\mathbb{R}^d:\langle \xi_\alpha , x \rangle \leq b_\alpha \}.$
\item $K^\circ = \bigcap_{\alpha\in\{-1,1\}^d} \{x\in\mathbb{R}^d:\langle \xi_\alpha , x \rangle < b_\alpha \}.$
\item $\partial K = \bigcup_{\alpha\in\{-1,1\}^d} \conv{x_{\alpha^1 1},\ldots,x_{\alpha^d d}}.$
\item {\small $\partial K = \left( \bigcup_{\alpha\in\{-1,1\}^d} \{x\in\mathbb{R}^d:\langle \xi_\alpha , x \rangle = b_\alpha \}\right)\bigcap \left(\bigcap_{\alpha\in\{-1,1\}^d} \{x\in\mathbb{R}^d:\langle \xi_\alpha , x \rangle \leq b_\alpha \}\right).$}
\item $\displaystyle B(0,\rho_*)\subset K^\circ .$
\item $\displaystyle \partial B(0,\rho^*)\subset \emph{Ext}(K).$
\end{enumerate}
\end{lemma}
Figure \ref{figura1}b illustrates Lemma \ref{l13} for the two-dimensional case. Intuitively, the idea is that as long as the points $x_{\pm 1}$ and $x_{\pm 2}$ belong to $B(\pm r\textbf{e}_1,\rho)$ and $B(\pm r\textbf{e}_2,\rho)$, respectively, we will have $B(0,\rho_*)$ and $\partial B(0,\rho^*)$ as subsets of $K^\circ$ and $\textrm{Ext}(K)$, respectively.
\begin{figure}[h!]
\centering \includegraphics[height=6cm,width=6.0cm]{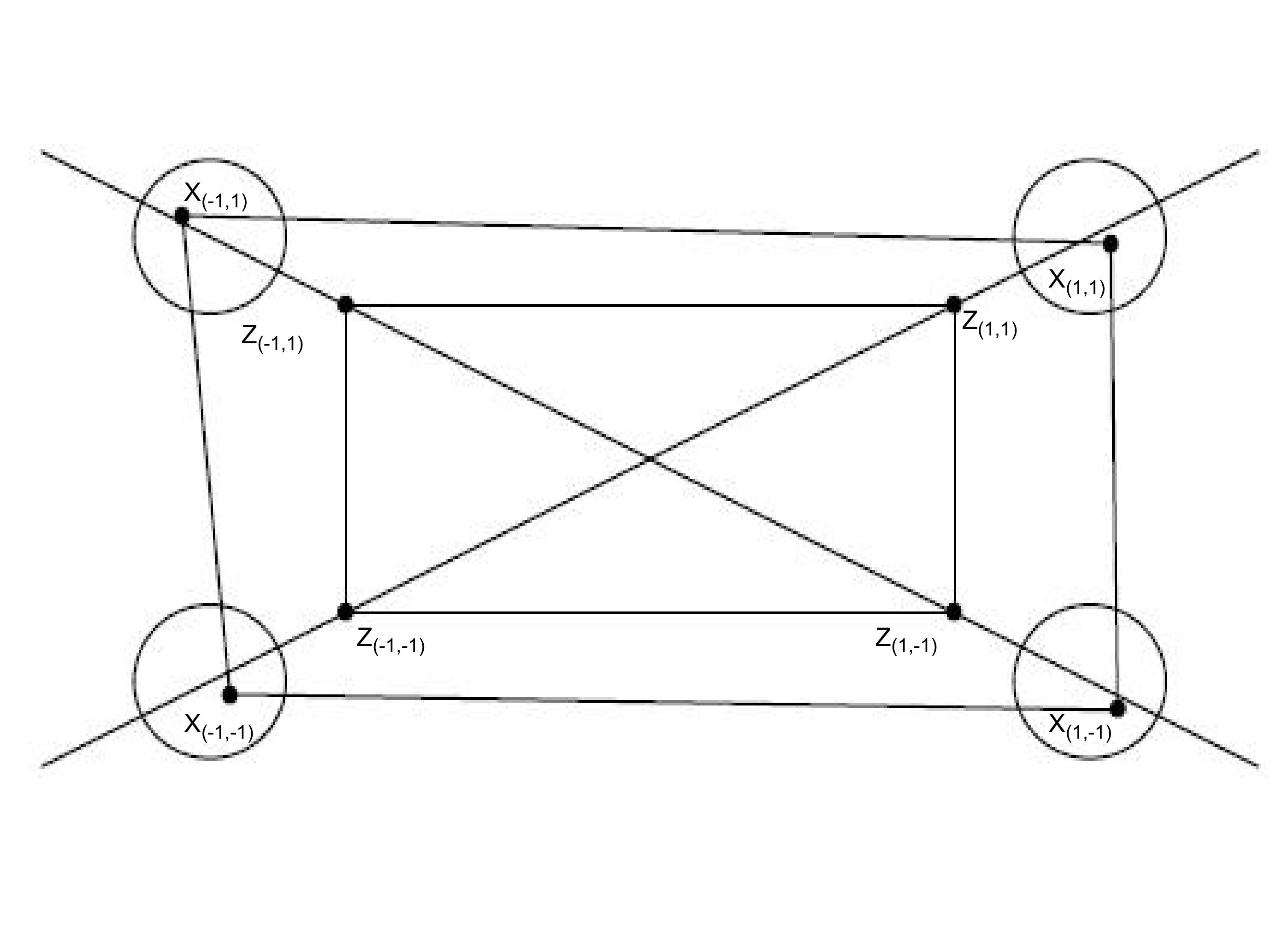} \includegraphics[height=6cm,width=6.5cm]{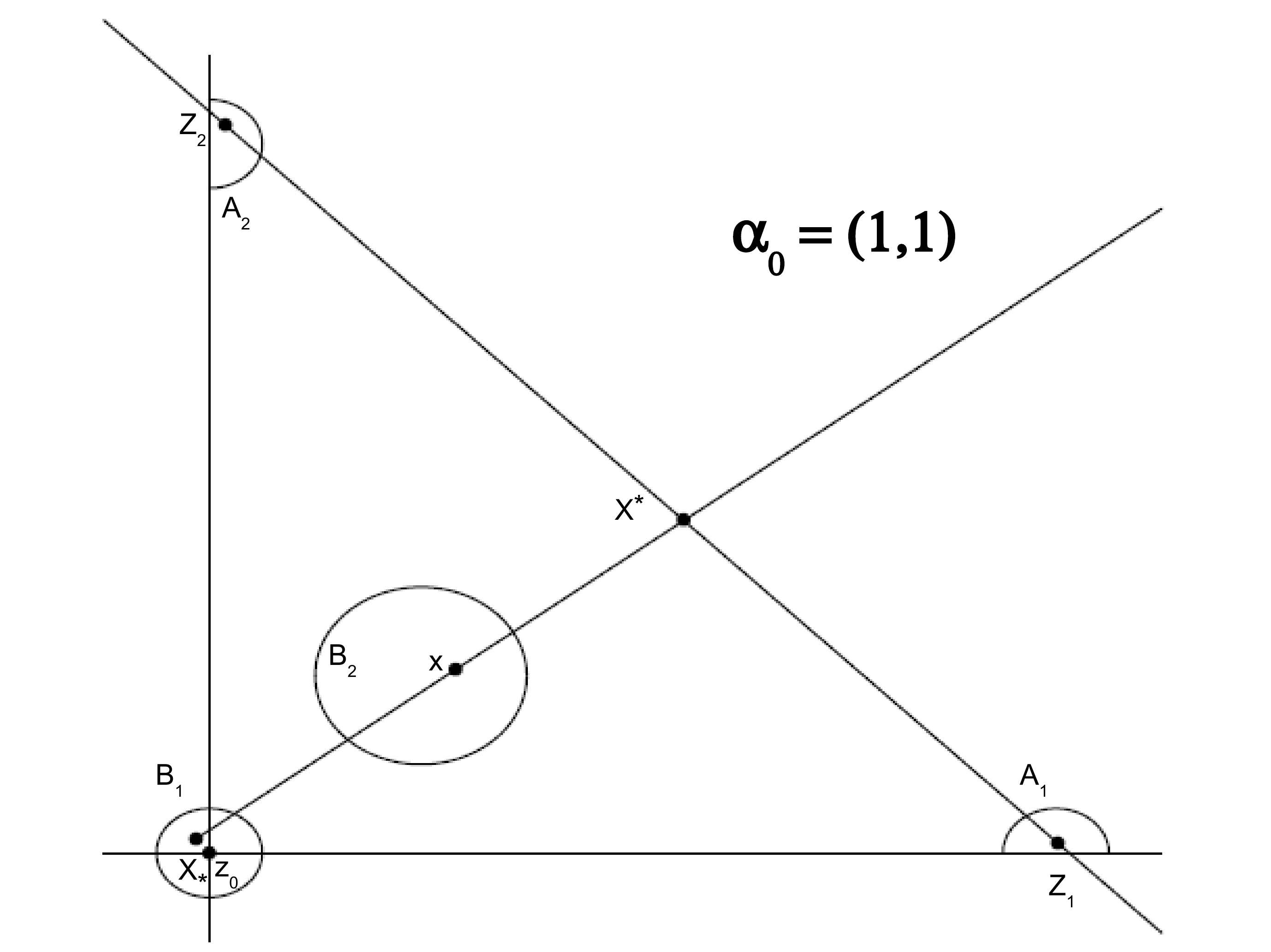}
\caption{Explanatory diagram for (a) Lemma \ref{l14} (left panel); (b) Lemma \ref{l16} (right panel).}\label{figura5}
\end{figure}
\begin{lemma}\label{l14}
Let $[a,b]\subset\mathbb{R}^d$ be a compact rectangle and $r>0$, with $r<\frac{1}{d-2}$ if $d\geq 3$. For each $\alpha\in\{-1,1\}^d$ write $z_\alpha = a+ \sum_{j=1}^d \frac{1+\alpha^j}{2}\left(b^j - a^j\right)\textbf{e}_j$ so that $\{z_\alpha\}_{\alpha\in\{-1,1\}^d}$ is the set of vertices of $[a,b]$. Then, there is $\rho > 0$ such that if $x_\alpha \in B(z_\alpha + r(z_\alpha - z_{-\alpha}),\rho)$ $\forall$ $\alpha\in\{-1,1\}^d$, then
\[ [a,b]\subset\conv{x_\alpha : \alpha\in\{-1,1\}^d}^\circ .\]
\end{lemma}
Figure \ref{figura5}a describes Lemma \ref{l14} in the two-dimensional case. As long as the points $x_{(\pm 1,\pm 1)}$ are chosen in the balls of radius $\rho$ around $z_{(\pm 1,\pm 1)} + r(z_{(\pm 1,\pm 1)} - z_{(\mp 1,\mp 1)})$, $\conv{x_{(\pm 1,\pm 1)}}$ will contain $\conv{z_{(\pm 1,\pm 1)}}$.
\subsection{Proof of Lemma \ref{l16}}\label{pl16}
Since any compact subset of $\mathfrak{X}^\circ$ is contained in a finite union of compact rectangles, it is enough to prove the result when $\texttt{X}$ is a compact rectangle $[a,b]\subset\mathfrak{X}^\circ$. Let $r =\frac{1}{4}\min_{1\leq k\leq d}\{b^k - a^k \}$ and choose $\rho\in(0,\frac{1}{4}r)$, $\rho^* > 0$ and $0 < \rho_* < \frac{1}{2}r$ such that the conclusions of Lemmas \ref{l12} and \ref{l13} hold for any $\alpha\in\{-1,1\}^d$ and any $\beta = (z_1,\ldots,z_d)\in\mathbb{R}^{d\times d}$ with $z_j \in B(\alpha^j r\textbf{e}_j,\rho)$. Take $N\in\mathbb{N}$ such that \begin{equation}\label{ec22}
\frac{1}{N} \max_{1\leq k\leq d} \{ b^k - a^k \}< \frac{1}{32 d}\rho_*
\end{equation}
and divide \texttt{X} into $N^d$ rectangles all of which are geometrically identical to $\frac{1}{N}[0,b-a]$. Let $\mathcal{C}$ be any one of the rectangles in the grid and choose any vertex $z_0$ of $\mathcal{C}$ satisfying
\begin{equation}\nonumber
z_0 = \argmax_{z\in\mathcal{C}}\left\{\max_{1\leq j \leq d} \left\{ z^j - a^j , b^j - z^j \right\}\right\}.
\end{equation}
Then, from the definition of $z_0$ and $r$, there is $\alpha_0\in\{-1,1\}^d$ such that
\begin{equation}\nonumber
B(z_0, r) \cap\left(z_0+\mathcal{R}_{\alpha_0}\right)\subset\texttt{X}.
\end{equation}
Additionally, define
\begin{eqnarray}
B_1 & = & B\left(z_0, \frac{\rho_*}{16\sqrt{d}}\right), \nonumber \\
B_2 & = & B \left(z_0 + \frac{3 \rho_*}{8\sqrt{d}}\alpha_0, \frac{\rho_*}{8\sqrt{d}}\right), \nonumber\\
A_{j} & = & B(z_0 + \alpha_0^j r\textbf{e}_j, \rho) \cap (z_0 + \mathcal{R}_{\alpha_0})\ \ \ \forall \ j=1,\ldots,d, \nonumber \\
A_{-j} & = & B(z_0 - \alpha_0^j r\textbf{e}_j, \rho)\ \ \  \forall \ j=1,\ldots,d. \nonumber
\end{eqnarray}
Observe that all the sets in the previous display have positive Lebesgue measure and that the $A_{-j}$'s are not necessarily contained in $\texttt{X}$. Let $M_1 = \left\|\phi\right\|_\texttt{X}$, $M_0 > \frac{\sigma}{\sqrt{\min\{ \nu(B_1), \nu(B_2), \nu(A_{1}),\ldots,\nu(A_{d})\}}}$, $M= M_1 + M_0$ and $K_\mathcal{C} > 6M$. Also, notice that $\mathcal{C}\subset B_1$ because of (\ref{ec22}). We will argue that
\begin{equation}\label{eq:boundInf}
\p{ \inf_{x\in\mathcal{C}} \{\hat{\phi}_n (x)\} \leq - K_\mathcal{C} \textrm{ i.o.}} = 0.
\end{equation}
From Lemma \ref{l15}, we know that
\begin{equation}\label{ec26}
\p{\bigcap_{j=1}^d \left[ \inf_{x\in A_{j}} \left\{ \left|\hat{\phi}_n (x) - \phi(x)\right|\right\} < M_0 \textrm{ a.a.}\right] }=1,
\end{equation}
so there is, with probability one, $n_0 \in\mathbb{N}$ such that $\inf_{x\in A_{j}} \left\{ \left|\hat{\phi}_n (x) - \phi(x)\right|\right\} < M_0$ for any $n\geq n_0$ and any $j=1,\ldots,d$.

Assume that the event $\left[  \inf_{x\in\mathcal{C}} \{\hat{\phi}_n (x)\}< - K_\mathcal{C} \textrm{ i.o.}\right]$ is true. Then, there is a subsequence $n_k$ such that $\inf_{x\in\mathcal{C}} \{\hat{\phi}_{n_k} (x)\}< - K_\mathcal{C}$ for all $k\in\mathbb{N}$.
Fix any $k \geq n_0$. We know that there is $X_* \in \mathcal{C}\subset B_1$ such that
$\hat{\phi}_{n_k}(X_*) \leq -K_\mathcal{C}$. In addition, for $j=1,\ldots,d$, there are $Z_{\alpha_0^j j } \in A_{j }$ such that $|\hat{\phi}_{n_k}(Z_{\alpha_0^j j}) - \phi(Z_{\alpha_0^j j})|<M_0$, which in turn implies $\hat{\phi}_{n_k}(Z_{\alpha_0^j j}) < M$. Pick any $Z_{-\alpha_0^j}\in A_{-j}$ and let $K=\conv{Z_{\pm 1},\dots, Z_{\pm d}}=z_0 + \conv{Z_{\pm 1}-z_0,\dots, Z_{\pm d}-z_0}$.

Take any $x\in B_2$. We will show the existence of $X^*\in \conv{Z_{\alpha_0^1 1},\ldots,Z_{\alpha_0^d d}}$ such that $x\in\conv{X_*,X^*}$, as shown in Figure \ref{figura5}b for the case $d=2$. We will then show that the existence of such an $X^*$ implies that
\begin{equation}\label{ec49}
|\phi(x)-\hat{\phi}_{n_k}(x)|>M_0.
\end{equation}
Consequently, since $x$ is an arbitrary element of $B_2$ we will have
\begin{equation}\nonumber
\left[\inf_{x\in\mathcal{C}} \{\hat{\phi}_n (x)\} \leq - K_\mathcal{C} \textrm{ i.o.}\right] \cap \left(\bigcap_{j=1}^d \left[ \inf_{x\in A_j} \left\{ \left|\hat{\phi}_n (x) - \phi(x)\right|\right\} < M_0 \textrm{ a.a.}\right]\right)
\end{equation}
\begin{equation}\nonumber
\subset\left[\inf_{x\in B_2} \{|\phi(x) - \hat{\phi}_{n_k}(x)|\}\geq M_0 \textrm{ i.o.} \right].
\end{equation}
But from Lemma \ref{l15}, the event on the right is a null set. Taking (\ref{ec26}) into account, we will see that (\ref{eq:boundInf}) holds and then complete the argument by taking $K_\texttt{X} = \max_\mathcal{C} \{K_\mathcal{C}\}$.

To show the existence of $X^*$ consider the function $\psi:\mathbb{R}\rightarrow\mathbb{R}^d$ given by $\psi(t) = X_* + t(x-X_*)$. The function $\psi$ is clearly continuous and satisfies $\psi(0) = X_*$ and $\psi(1) = x\in B_2 \subset K^\circ$. That $B_2\subset K^\circ$ is a consequence of Lemma \ref{l12}, $(iv)$. The set $K$ is bounded, so there is $T > 1$ such that $\psi (T) \in \textrm{Ext} (K) = \mathbb{R}^d \setminus \overline{K}$. The intermediate value theorem then implies that there is $t^* \in (1,T)$ such that $X^*:= \psi(t^*) \in \partial K$. Observe that by Lemma \ref{l13} $(iii)$ we have
\[ \partial K = \bigcup_{\alpha\in\{-1,1\}^d}\conv{Z_{\alpha^1 1},\ldots, Z_{\alpha^d d} }. \]
Lemma \ref{l12} $(i)$ implies that $\{Z_{\alpha_0^1 1}-z_0,\ldots,Z_{\alpha_0^d d}-z_0\}$ forms a basis of $\mathbb{R}^d$ so we can write $X^*-z_0 = \sum_{j=1}^d \theta^j (Z_{\alpha_0^j j}-z_0)$. Moreover, Lemma \ref{l12} $(vii)$ implies that $\theta^j > 0$ for every $j=1,\ldots,d$ as $\theta = (\theta^1,\ldots,\theta^d) = (Z_{\alpha_0^1 1}-z_0,\ldots,Z_{\alpha_0^d d}-z_0)^{-1}(X^*-z_0)$. Here we apply Lemma \ref{l12} $(vii)$ with $w_1 = X_*\in B_1$, $w_2 = x\in B_2$ and $t^* > 1$.

For $\alpha \in \{-1,1\}^d$ consider the pair $(\xi_\alpha,b_\alpha)\in\mathbb{R}^d\times\mathbb{R}$ as defined in Lemma \ref{l13} for the set of vectors $\{Z_{\pm 1}-z_0,\ldots,Z_{\pm d}-z_0\}$ (here we move the origin to $z_0$). Observe that Lemma \ref{l12} $(ii)$ implies that $\langle \xi_{\alpha_0}, Z_{\alpha_0^j j} - z_0\rangle = b_{\alpha_0}$ for all $j=1,\ldots, d$. Consequently, $\langle \xi_{\alpha_0},X^*-z_0 \rangle = b_{\alpha_0}\sum_{j=1}^d \theta^j$, but since $X^*\in\partial K$, Lemma \ref{l13} ($iv$) implies that $\langle \xi_{\alpha_0},X^*-z_0 \rangle \le b_{\alpha_0}$ and hence $\sum_{j=1}^d \theta^j \leq 1$. Additionally, for $\alpha\neq\alpha_0$ we can write $\langle \xi_\alpha, X^*-z_0 \rangle$ as
\begin{equation}\label{ec45}
\sum_{j=1}^d \theta^j  \langle \xi_\alpha, Z_{\alpha^j_0 j}-z_0\rangle = \sum_{\alpha^j =\alpha_0^j} \theta^j b_\alpha + \sum_{\alpha^j \neq \alpha_0^j}\theta^j \langle \xi_\alpha, Z_{\alpha_0^j j}-z_0\rangle < b_\alpha
\end{equation}
as $\langle \xi_\alpha , Z_{\alpha^j} -z_0\rangle = b_\alpha$ (by Lemma \ref{l12} $(ii)$) and $\langle \xi_\alpha , Z_{-\alpha^j} -z_0 \rangle < 0$ (by Lemma \ref{l12} $(vi)$) for every $j=1,\ldots,d$. Since $\langle \xi_\alpha, w -z_0 \rangle = b_\alpha$ for all $w\in\conv{Z_{\alpha^1 1},\ldots,Z_{\alpha^d d}}$ and all $\alpha\in\{-1.1\}^d$, (\ref{ec45}) and the fact that $X^*\in\partial K$ imply that $X^*\in\conv{Z_{\alpha_0^1 1},\ldots,Z_{\alpha_0^d d}}$. Hence $\hat{\phi}_n(X^*)  \leq \sum_{j=1}^d \theta^j\hat{\phi}_{n_k} ( Z_{\alpha_0^j j}) < M$. We therefore have
\begin{eqnarray}
\hat{\phi}_{n_k}(X^*) < M &,& \hat{\phi}_{n_k}(X_*) < -K_\mathcal{C},\label{ec24}\\
 X_* + \frac{1}{t^*}(X^* - X_*) &=& x.\label{ec25}
\end{eqnarray}
Since $X_*\in B_1$ and $d\geq 1$ we have
\begin{equation}\label{ec46}
|z_0 - X_*| < \frac{1}{8}\rho_* .
\end{equation}
By using the triangle inequality we get the following bounds
\begin{equation}
\frac{1}{4}\rho_* < |z_0 - x| < \frac{1}{2}\rho_*.\label{ec47}
\end{equation}
And from Lemma \ref{l12} $(iv)$ and the fact that $\langle \xi_{\alpha_0} , X^* \rangle = b_{\alpha_0}$ we also obtain
\begin{equation}
|z_0 - X^*| \geq \rho_* .\label{ec48}
\end{equation}
From (\ref{ec25}) we know that $t^* = \frac{|X^* - X_*|}{|x-X_*|}$. Using the triangle inequality with (\ref{ec46}), (\ref{ec47}) and (\ref{ec48}) one can find lower and upper bounds for $|X^* - X_*|$ (as $|X^* - X_*|\geq |X^* - z_0| - |z_0 - X_*|$) and $|x-X_*|$ (as $|x - X_*|\leq |x - z_0| + |z_0 - X_*|$), respectively, to obtain $t^*\geq \frac{7}{5}$. Then, (\ref{ec24}) and (\ref{ec25}) imply
\[ \hat{\phi}_{n_k}(x)   \leq  \left( 1 - \frac{1}{t^*} \right) \hat{\phi}_{n_k}(X_*) + \frac{1}{t^*} \hat{\phi}_{n_k}(X^*) \leq -\frac{2}{7}K_\mathcal{C} + \frac{5}{7}M < -M. \]
Consequently,
\[ |\phi(x) - \hat{\phi}_{n_k}(x)| > M - M_1 = M_0. \]
This proves (\ref{ec49}) and completes the proof. $\hfill \square$

\subsection{Proof of Lemma \ref{l17}}\label{pl17}
Assume without loss of generality that $\texttt{X}$ is a compact rectangle. Let $\{z_\alpha: \alpha\in\{-1,1\}^d \}$ be the set of vertices of the rectangle. Then, there is $r\in (0,1)$ such that $B(z_\alpha,r)\subset\mathfrak{X}^\circ$ $\forall$ $\alpha\in\{-1,1\}^d$. Recall that from Lemma \ref{l14}, there is $0< \rho < \frac{1}{2}r$ such that for any $\{\eta_\alpha : \alpha\in\{-1,1\}^d \}$ if $\eta_\alpha\in B(z_\alpha + \frac{r}{2}(z_\alpha - z_{-\alpha}),\rho)$ then $\texttt{X}\subset \conv{\eta_\alpha : \alpha\in\{-1,1\}^d }$.

Let $A_\alpha = B(z_\alpha + \frac{1}{2}r(z_\alpha - z_{-\alpha}),\frac{\rho}{2})$ and $M_0 > \frac{\sigma}{\sqrt{\min\{\nu(A_\alpha) : \alpha\in\{-1,1\}^d\}}}$ and choose
\[ M_1 = \sup_{x\in \conv{\bigcup_{\alpha\in\{-1,1\}^d} A_\alpha}} \{ |\phi(x) | \}. \]
Take $K_\texttt{X} > M_0 + M_1$. Since
\[ \p{\bigcap_{\alpha\in\{-1,1\}^d} \left[ \inf_{x\in A_\alpha} \{|\hat{\phi}_n (x) - \phi(x)|\} < M_0, \textrm{ a.a.} \right]} = 1 \]
by Lemma \ref{l15}, there is, with probability one, $n_0\in\mathbb{N}$ such that for any $n\geq n_0$ we can find $\eta_\alpha\in A_\alpha$, $\alpha\in\{-1,1\}^d$, such that $|\hat{\phi}_n (\eta_\alpha) - \phi(\eta_\alpha)| < M_0$. It follows that $\hat{\phi}_n (\eta_\alpha) \leq K_\texttt{X}$ $\forall$ $\alpha\in\{-1,1\}^d$. Now, using Lemma \ref{l14} we have $\texttt{X} \subset \conv{\eta_\alpha : \alpha\in\{-1,1\}^d}$ and the convexity of $\hat{\phi}_n$ implies that $\hat{\phi}_n(x) \leq K_\texttt{X}$ for any $x\in\texttt{X}$. $\hfill \square$

\subsection{Proof of Lemma \ref{l18}}\label{pl18}
Assume that $\texttt{X}=[a,b]$ is a rectangle with vertices $\{z_\alpha:\alpha\in\{-1,1\}^d\}$. The function $\psi(x) = \inf_{\eta\in\overline{\textrm{Ext}(\mathfrak{X})}}\{|x-\eta|\}$ is continuous on $\mathbb{R}^d$ so there is $x_* \in \partial \texttt{X}$ such that $\psi(x_*) = \inf_{x\in\partial\texttt{X}}\{\psi(x)\}$. Observe that $\psi(x_*)>0$ because $x_*\in\partial \texttt{X}\subset\mathfrak{X}^\circ$. By Lemma \ref{l14}, there is a $r < \frac{1}{2}\psi(x_*)$ for which there exists $\rho < \frac{1}{4} r$ such that whenever $\eta_\alpha\in A_\alpha := B\left(z_\alpha + \frac{3}{4}r\left(\frac{z_\alpha - z_{-\alpha}}{|z_\alpha - z_{-\alpha}|}\right),\rho\right)$ for any $\alpha\in\{-1,1\}^d$ and
\begin{eqnarray}
K_z &=& \conv{z_\alpha + \frac{1}{2}r\left(\frac{z_\alpha - z_{-\alpha}}{|z_\alpha - z_{-\alpha}|}\right):\alpha\in\{-1,1\}^d}\nonumber\\
K_\eta &=& \conv{\eta_\alpha:\alpha\in\{-1,1\}^d}\nonumber
\end{eqnarray}
we have
\begin{eqnarray}\label{ec27}
\texttt{X}\subset K_z \subset K_\eta^\circ \subset K_\eta \subset\mathfrak{X}^\circ .
\end{eqnarray}
Let $M_0 > \frac{\sigma}{\sqrt{\min\{\nu(A_\alpha) : \alpha\in\{-1,1\}^d\}}}$ and $M_1 \in \mathbb{R}$ be such that
 \[ \p{ \inf_{x\in\texttt{X}} \{\hat{\phi}_n (x)\} \leq - M_0 \textrm{ i.o.}} = 0 \;\; \mbox{ and } \;\; M_1 = \sup_{x\in\conv{\bigcup_{\alpha\in\{-1,1\}^d} A_\alpha}}\{\phi(x)\}. \]
From Lemmas \ref{l15} and \ref{l16} we can find, with probability one, $n_0\in\mathbb{N}$ such that $\inf_{x\in\texttt{X}}\{\hat{\phi}_n(x)\} > -M_0$ and $\inf_{x\in A_\alpha} \{ |\hat{\phi}_n(x) - \phi(x)|\} < M_0$ for any $n\geq n_0$. Define
\begin{eqnarray}
M &=& M_1 + M_0\nonumber\\
K_\texttt{X} &=& \frac{4|b-a|}{r\min_{1\leq j \leq d}\{b^j - a^j\}}M\nonumber
\end{eqnarray}
and take any $n\geq n_0$. Then, for any $\alpha\in\{-1,1\}^d$ we can find $\eta_\alpha\in A_\alpha$ such that $|\hat{\phi}_n (\eta_\alpha) - \phi(\eta_\alpha)| < M_0$. Then, (\ref{ec27}) implies that $\hat{\phi}_n (x)\leq M$ $\forall x\in\texttt{X}$. Take then $x\in\texttt{X}$ and $\xi\in\partial \hat{\phi}_n (x)$. A connectedness argument, like the one used in the proof of Lemma \ref{l16}, implies that there is $t_* > 0$ such that $x+t_*\xi\in\partial K_\eta$. But then we must have $t_* > \frac{r\min_{1\leq j\leq d}\{b^j - a^j\}}{2 |\xi||b-a|}$ as a consequence of (\ref{ec27}), since the smallest distance between $\partial K_z$ and $\partial \texttt{X}$ is $\frac{r\min_{1\leq j\leq d}\{b^j - a^j\}}{2 |b-a|}$ and $\partial K_\eta\subset\textrm{Ext}(K_z)$. This can be seen by taking a look at Figure \ref{figura4}, which shows the situation in the two dimensional case.
\begin{figure}
\centering
\includegraphics[height=7cm,width=9cm]{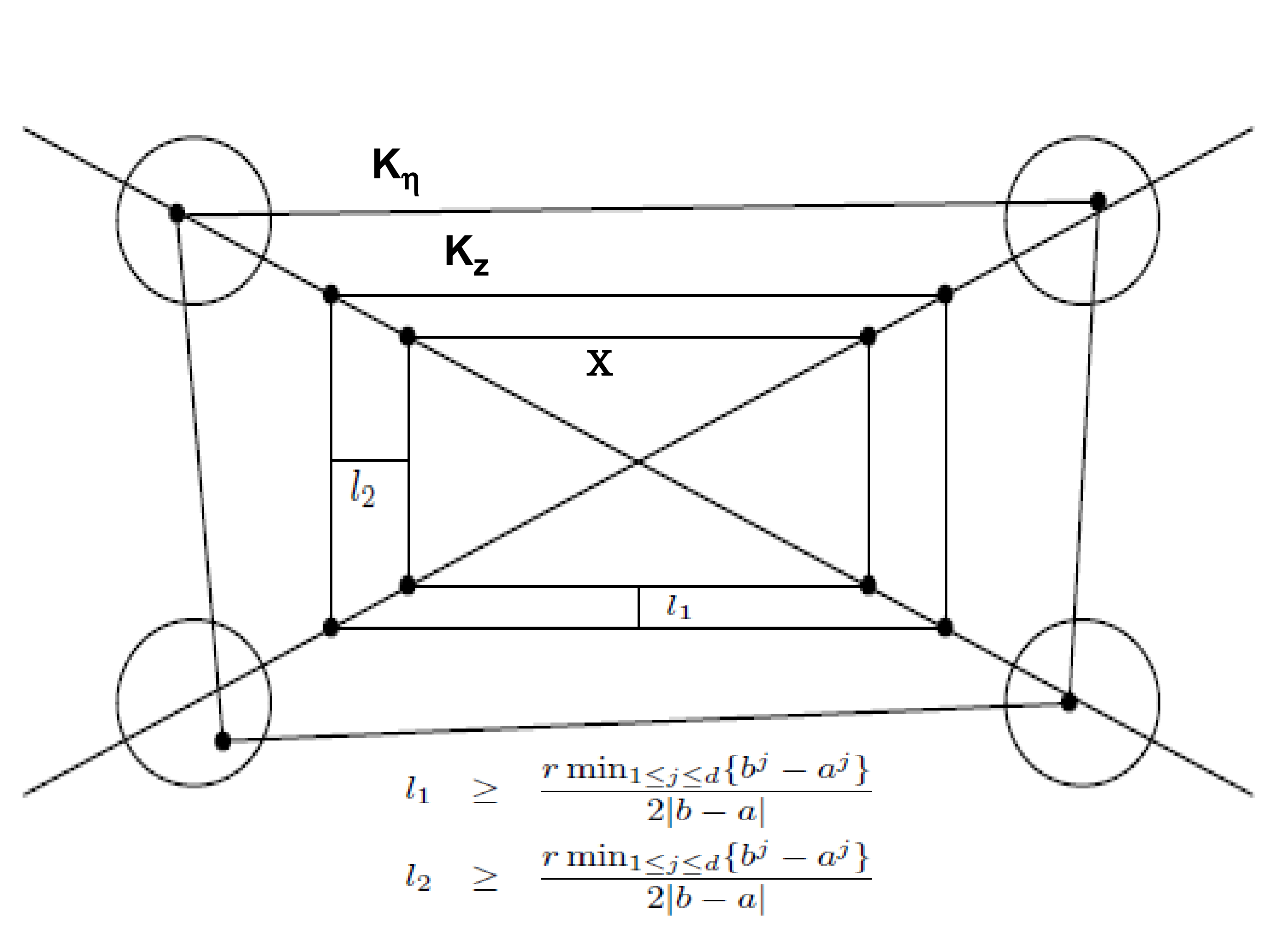}
\caption{The smallest distance between $\partial K_z$ and $\partial \texttt{X}$ is at least $\frac{r\min_{1\leq j\leq d}\{b^j - a^j\}}{2 |b-a|}$.}\label{figura4}
\end{figure}
Thus, using the definition of subgradients,
\[ \frac{r\min_{1\leq j\leq d}\{b^j - a^j\}}{2|\xi||b-a|}\langle\xi,\xi\rangle \leq  \langle\xi,t_* \xi\rangle \leq \hat{\phi}_n(x+t_*\xi) - \hat{\phi}_n(x) \leq 2M \]
which in turn implies $|\xi| \leq K_\texttt{X}$. We have therefore shown that, with probability one, we can find $n_0\in\mathbb{N}$ such that $|\xi| \leq K_\texttt{X}$ $\forall$ $\xi\in\partial \hat{\phi}_n(x)$, $\forall$ $x\in\texttt{X}$, $\forall$ $n\geq n_0$. This completes the proof. $\hfill \square$

\subsection{Proof of Lemma \ref{l19}}\label{pl19}
The result is obvious for conditions \{A1-A3\} and \{A5-A7\} when $\sigma^2 =0$. So we assume that $\sigma^2 > 0$ for \{A1-A3\} and \{A5-A7\}. Let $\epsilon > 0$ and $M=\sup_{x\in\texttt{X}}\{|x|\}$. Choose $\delta > 0$ satisfying
\begin{equation}\label{ec28}
\small \frac{\epsilon}{\frac{2(2M+K\sqrt{d}+1)}{n}\sum_{j=1}^n |Y_j - \phi(X_j)|}< \delta < \frac{\epsilon}{\frac{(2M+K\sqrt{d}+1)}{n}\sum_{j=1}^n |Y_j - \phi(X_j)|}
\end{equation}
for $n$ large. Notice that $\delta$ is well-defined and the quantity on the left is positive, finite and bounded away from 0 as $\linf \frac{1}{n}\sum_{j=1}^n |Y_j - \phi(X_j)| > 0$ a.s. under any set of regularity conditions (for \{A2-A4\}, conditions A4-(i) and A4-(iii) imply that we can apply the version of the strong law of large number for uncorrelated random variables, as it appears in \cite{chung}, page 108, Theorem 5.1.2 to the sequence $(|\epsilon_j|)_{j=1}^\infty$; for \{A1-A3\} and \{A5-A7\} this is immediate as $\sigma^2 >0$). The definition of the class $\mathcal{D}_{K,\texttt{X}}$ implies that all its members are Lipschitz functions with Lipschitz constant bounded by $K\sqrt{d}$, a consequence of \cite{rca}, Theorem 24.7, page 237. Hence, (\ref{ec28}) implies that
\begin{equation}
\sup_{\begin{subarray}{c} |x-y| < \delta\\ x,y\in\texttt{X}, \psi\in\mathcal{D}_{K,\texttt{X}}\end{subarray}}\{|\psi(x) - \psi(y)|\} \leq \frac{\epsilon}{\frac{1}{n}\sum_{j=1}^n |Y_j - \phi(X_j)|}. \nonumber
\end{equation}
Now, define $N_n\in\mathbb{N}$ by $N_n = \left\lceil  \frac{\textrm{diam}(\texttt{X})}{\delta}\right\rceil \lor \left \lceil \frac{2 K\sqrt{d}}{\delta} \right \rceil$, where $\lceil \cdot \rceil$ denotes the ceiling function.
Observe that (\ref{ec28}) implies
\begin{equation}\label{ec35}\small
N_n -1 \leq \left(\textrm{diam}(\texttt{X})\lor 2 K\sqrt{d}\right)\frac{2(2M+K\sqrt{d}+1)}{\epsilon} \left(\frac{1}{n}\sum_{j=1}^n |Y_j - \phi(X_j)| \right).
\end{equation}
Then, we can divide the rectangles $\texttt{X}$ and $[-K,K]^d$ in $N_n^d$ subrectangles, all of which have diameters less than $\delta$. In other words, we can write
\begin{eqnarray}
[-K,K]^d &=& \bigcup_{1\leq j\leq N_n^d} R_j\nonumber\\
\texttt{X} &=& \bigcup_{1\leq j\leq N_n^d} V_j\nonumber
\end{eqnarray}
with $\textrm{diam}(R_j)<\delta$ and $\textrm{diam}(V_j)<\delta$ $\forall$ $j=1,\ldots N_n^d$. In the same way, we can divide the interval $[-K,K]$ in $N_n$ subintervals $\mathcal{I}_1,\ldots,\mathcal{I}_{N_n}$ each having length less than $\delta$. For each $j=1,\ldots,N_n^d$, let $\xi_j$ and $x_j$ be the centroids of $R_j$ and $V_j$ respectively and for $j=1,\ldots,N_n$ let $\eta_j$ be the midpoint of $\mathcal{I}_j$. Consider the class of functions $\mathcal{H}_{n,\epsilon}$ defined by
\begin{equation}\nonumber
\mathcal{H}_{n,\epsilon} = \left\{ \max_{(s,t,j)\in\mathcal{S}}\{ \langle \xi_s , \cdot - x_t\rangle + \eta_j \} : \mathcal{S}\subset\{1,\ldots,N_n^d\}^2\times\{1,\ldots,N_n\}\right\}.
\end{equation}
Observe that the number of elements in the class $\mathcal{H}_{n,\epsilon}$ is bounded from above by $2^{N_n^{2d+1}}$. Now, take any $\psi\in\mathcal{D}_{K,\texttt{X}}$. Pick any $\Xi_j\in\partial\psi(X_j)$. Then, for any $j$ such that $X_j\in\texttt{X}$, there are $s_j,t_j\in\{1,\ldots,N_n^d\}$ and $\tau_j\in\{1,\ldots,N_n\}$ such that $|\Xi_j - \xi_{s_j}|$, $|X_j - x_{t_j}|$ and $|\psi(x_{t_j}) - \eta_{\tau_j}|$ are all less than $\delta$. We then have that
\begin{eqnarray}
& & \sup_{x\in\texttt{X}} \left\{ \left|\langle \xi_{s_j}, x - x_{t_j}\rangle + \eta_{\tau_j} - \left(\langle \Xi_j, x - X_j\rangle + \psi(X_j)\right)\right|\right\}  \nonumber \\
& \leq & 2 M|\xi_{s_j} - \Xi_j| + K\sqrt{d}|x_{t_j} - X_j| + \delta < (2M+K\sqrt{d}+1)\delta \label{ecimportante}
\end{eqnarray}
by an application of the Cauchy-Schwarz inequality. But then, (\ref{ec28}) implies that if we define the functions $\tilde{\psi}$ and $g$ as
\begin{eqnarray}
\tilde{\psi}(x) &=& \max_{X_j\in\texttt{X}} \{ \langle \Xi_j,x -X_j \rangle + \psi(X_j) \},\nonumber\\
g(x) &=& \max_{X_j\in\texttt{X}} \{ \langle \xi_{s_j},x -x_{t_j} \rangle + \eta_{\tau_j} \}\nonumber
\end{eqnarray}
then we have
\begin{eqnarray}
\tilde{\psi}(X_j) &=& \psi(X_j)\ \textrm{ for } j \textrm{ such that }X_j\in\texttt{X}, \label{ec30}\\
\| g - \tilde{\psi} \|_\texttt{X} &\leq& \frac{\epsilon}{\frac{1}{n} \sum_{j=1}^n |Y_j - \phi(X_j)|}\textrm { (from (\ref{ecimportante}))},\label{ec31}\\
g &\in& \mathcal{H}_{n,\epsilon}.\label{ec32}
\end{eqnarray}
Note that (\ref{ec30}) follows from the definition of subgradients. All these facts put together give that for any $f(x,y) = \psi(x)(y-\phi(x))\in\mathcal{G}_{K,\texttt{X}}$, $\psi\in\mathcal{D}_{K,\texttt{X}}$ there is $g\in\mathcal{H}_{n,\epsilon}$ such that
\[ \int_\texttt{X} |f(x,y) - g(x)(y-\phi(x))|\mu_n (dx,dy) < \epsilon \]
and hence
\[ N(\epsilon, \mathcal{G}_{K,\texttt{X}}, \mathbb{L}_1 (\texttt{X}\times\mathbb{R},\mu_n))\leq \# \mathcal{H}_{n,\epsilon} \leq 2^{N_n^{2d + 1}}. \]
But then, the strong law of large numbers and (\ref{ec35}) give that $\lsup N_n <\infty$ a.s. Furthermore, by replacing $\epsilon$ with $\frac{\epsilon}{n}\sum_{j=1}^n|Y_j - \phi(X_j)|$ in the entire construction just made, we can see that the covering numbers \\ $N\left(\frac{\epsilon}{n}\sum_{j=1}^n|Y_j - \phi(X_j)|, \mathcal{G}_{K,\texttt{X}}, \mathbb{L}_1 (\texttt{X}\times\mathbb{R},\mu_n)\right)$ depend neither on the $Y$'s nor on $\phi$. Taking {\footnotesize $B_\epsilon = \left(\textrm{diam}(\texttt{X})\lor K\sqrt{d}\right)\frac{2(2M+K\sqrt{d}+1)}{\epsilon} + 1$} and $A_\epsilon = 2^{B_\epsilon^{2d+1}}$ it is seen that the second part of the result holds. $\hfill \square$

\subsection{Proof of Lemma \ref{lfdp}}\label{plfdp}
Note that for every $m$, we have
\begin{equation}
\frac{1}{n_k}\sum_{1\leq j\leq n_k}\e{\epsilon_j^2} \leq \frac{1}{n_k}\sum_{\begin{subarray}{c}X_j\in\texttt{X}_m\\ 1\leq j\leq n_k\end{subarray}}\e{\epsilon_j^2} + \frac{N_{n_k}(\mathfrak{X}\setminus\texttt{X}_m)}{n_k}\sup_{j\in\mathbb{N}}\{\e{\epsilon_j^2}\}.\nonumber
\end{equation}
Taking limit inferior on both sides as $k \rightarrow \infty$, we get
$$\sigma^2 \leq \linf_{k\rightarrow\infty}\frac{1}{n_k}\sum_{\begin{subarray}{c}X_j\in\texttt{X}_m\\ 1\leq j\leq n_k\end{subarray}}\e{\epsilon_j^2} + \nu(\mathfrak{X}\setminus\texttt{X}_m)\sup_{j\in\mathbb{N}}\{\e{\epsilon_j^2}\}.$$ Now taking the limit as $m\rightarrow\infty$ we get the result because the opposite inequality is trivial. $\hfill \square$

\subsection{Proof of Lemma \ref{l20}}\label{pl20}
We may assume that $\texttt{X}$ is a compact rectangle. Here we need to make a distinction between the design schemes. In the case of the stochastic design, the proof is an immediate consequence of Lemma \ref{l19} and Theorem 2.4.3, page 123 of \cite{vw}. Thus, we focus on the fixed design scenario.

For notational convenience, we write $M = \sup_{j\in\mathbb{N}}\{ \e{\epsilon_j^2}\}$ and $\sum_{X_j\in\texttt{X}}$ instead of the more cumbersome $\sum_{1\leq j\leq n: X_j\in\texttt{X}}$. Letting $\epsilon_j = Y_j - \phi(X_j)$ (and using the same notation as in the proof of Lemma \ref{l20}) first observe that the random quantity \[ \sup_{\psi\in\mathcal{D}_{K,\texttt{X}}} \left\{ \left|\frac{1}{n}\sum_{\{X_j\in\texttt{X}\}} \psi(X_j) \epsilon_j \right|\right\}  = \sup_{m\in\mathbb{N}} \left\{ \sup_{g\in\mathcal{H}_{n,\frac{1}{m}}} \left\{\left| \frac{1}{n}\sum_{\{X_j\in\texttt{X}\}} g(X_j) \epsilon_j \right|\right\}\right\}. \]
by (\ref{ec30}), (\ref{ec31}) and (\ref{ec32}) and is thus measurable.

All of the following arguments are valid for both, \{A1-A3\} and \{A2-A4\}. Lyapunov's inequality (which states that for any random variable $X$ and $1\leq p\leq q \leq \infty$ we have $\|X\|_p \leq \|X\|_q$) and the strong law of large numbers imply
\begin{equation}\label{ec63}
\lsup_{m\rightarrow\infty}\frac{1}{m}\sum_{1\leq j\leq m} |\epsilon_j | = \lsup_{m\rightarrow\infty}\frac{1}{m}\sum_{1\leq j\leq m} \e{|\epsilon_j |} \leq \sqrt{M}\textrm{ a.s.}
\end{equation}

Let $\eta > 0$. From Lemma \ref{l19} we know that the covering numbers {\footnotesize $a_n := N\left(\frac{\eta}{n}\sum_{j=1}^n|Y_j - \phi(X_j)|, \mathcal{G}_{K,\texttt{X}}, \mathbb{L}_1 (\texttt{X}\times\mathbb{R},\mu_n)\right)$}
are not random and uniformly bounded by a constant $A_\eta$. Therefore, for any $n\in\mathbb{N}$ we can find a class $\mathcal{A}_{n}\subset\mathcal{D}_{K,\texttt{X}}$ with exactly $a_n$ elements such that $\{\psi(x)(y-\phi(x))\}_{\psi\in\mathcal{A}_n}$ forms an $\left(\frac{\eta}{n}\sum_{j=1}^n|Y_j - \phi(X_j)|\right)$-net for $\mathcal{G}_{K,\texttt{X}}$ with respect to $\mathbb{L}_1 (\texttt{X}\times\mathbb{R},\mu_n)$. It follows that
\begin{equation}\label{ec36}
\small \sup_{\psi\in\mathcal{D}_{K,\texttt{X}}} \left\{ \left|\frac{1}{n}\sum_{X_j\in\texttt{X}} \psi(X_j) \epsilon_j \right|\right\}\leq \frac{\eta}{n}\sum_{1\leq j\leq n} |\epsilon_j | \ + \sup_{\psi\in\mathcal{A}_n} \left\{ \left|\frac{1}{n}\sum_{X_j\in\texttt{X}} \psi(X_j)\epsilon_j\right|\right\}.
\end{equation}
With (\ref{ec36}) in mind, we make the following definitions
\begin{eqnarray}
B_n &=& \sup_{\psi\in\mathcal{A}_n} \left\{ \left|\frac{1}{n}\sum_{X_j\in\texttt{X}} \psi(X_j)\epsilon_j  \right|\right\}, \nonumber \\
C_n &=& \sup_{\psi\in\mathcal{A}_n} \left\{\left|\frac{1}{n}\sum_{1 \leq j \leq \lfloor \sqrt{n} \rfloor^2 :\ X_j\in\texttt{X}} \psi(X_j)\epsilon_j \right|\right\},\nonumber\\
D_n &=& \sup_{\begin{subarray}{c} \psi\in\mathcal{A}_k\\ n^2 \leq k < (n+1)^2\end{subarray}}\left\{\left|\frac{1}{k}\sum_{n^2 < j \leq k :\ X_j\in\texttt{X}} \psi(X_j)\epsilon_j \right|\right\},\nonumber
\end{eqnarray}
where $\lfloor \cdot \rfloor$ denotes the floor function. Now, pick $\delta > 0$ and observe that
\begin{eqnarray}
\p{B_n > \delta} &=& \p{\bigcup_{\psi\in\mathcal{A}_n} \left[\left|\sum_{X_j\in\texttt{X}} \psi(X_j)\epsilon_j \right|> n\delta\right]}\nonumber\\
&\leq& \sum_{\psi\in\mathcal{A}_n} \frac{1}{n^2\delta^2}M\sum_{X_j\in\texttt{X}}\psi(X_j)^2 \; \; \leq \; \; \frac{K^2 M A_\eta}{n\delta^2}.\nonumber
\end{eqnarray}
The Borel-Cantelli Lemma then implies that $\p{B_{n^2} > \delta\ \textrm{i.o.}} = 0$. Letting $\delta\rightarrow 0$ through a decreasing sequence gives
\begin{equation}\label{ec37}
B_{n^2}\cas 0.
\end{equation}
On the other hand, the definition of $C_n$ implies that
\begin{equation}
C_n \leq \frac{\lfloor \sqrt{n}\rfloor ^2}{n}B_{\lfloor \sqrt{n} \rfloor ^2} +  \frac{\eta}{n}\sum_{1\leq j\leq \lfloor \sqrt{n}\rfloor^2} |\epsilon_j |\label{eq:Cn}
\end{equation}
which together with (\ref{ec37}) and (\ref{ec63}) gives
\begin{equation}\label{ec60}
\lsup C_n \leq \eta\sqrt{M}\ \textrm{ almost surely.}
\end{equation}
Note that (\ref{eq:Cn}) is a consequence of the fact that for any $\psi \in \mathcal{A}_n$, there exists $g \in \mathcal{A}_{\lfloor \sqrt{n} \rfloor^2}$ such that if $\mathcal{J}_n = \{1\leq j \leq \lfloor \sqrt{n} \rfloor^2 : X_j\in\texttt{X}\}$, then
\begin{eqnarray*}
\left|\frac{1}{n}\sum_{j\in\mathcal{J}_n} \psi(X_j)\epsilon_j \right| & \le & \left|\frac{1}{n}\sum_{j\in\mathcal{J}_n} (\psi(X_j) - g(X_j) )\epsilon_j \right| + \left|\frac{1}{n}\sum_{j\in\mathcal{J}_n} g(X_j)\epsilon_j \right| \\
& \le & \left(\frac{\lfloor \sqrt{n} \rfloor^2}{n}\right)\frac{\eta}{\lfloor \sqrt{n} \rfloor^2}\sum_{1\leq j\leq \lfloor \sqrt{n}\rfloor^2} |\epsilon_j | + \frac{\lfloor \sqrt{n}\rfloor ^2}{n}B_{\lfloor \sqrt{n} \rfloor ^2}.
\end{eqnarray*}
Now, a similar argument to the one used in (\ref{ec37}) gives
\begin{eqnarray}
\p{D_n > \delta} &=& \p{\bigcup_{\begin{subarray}{c} \psi\in\mathcal{A}_k\\ n^2 \leq k < (n+1)^2\end{subarray}} \left[\left|\sum_{n^2 < j \leq k : X_j\in\texttt{X}} \psi(X_j)\epsilon_j \right|> k \delta\right]}\nonumber \\
&\leq& \sum_{\begin{subarray}{c} \psi\in\mathcal{A}_k\\ n^2 \leq k < (n+1)^2\end{subarray}}\p{\left|\sum_{n^2 < j \leq k : X_j\in\texttt{X}} \psi(X_j)\epsilon_j \right|> k \delta}\nonumber\\
&\leq& \sum_{\begin{subarray}{c} \psi\in\mathcal{A}_k\\ n^2 \leq k < (n+1)^2\end{subarray}}\frac{K^2 M (k-n^2)}{k^2\delta^2} \; \; \leq \; \;  \frac{K^2M A_\eta (2n+1)^2}{n^4 \delta^2}.\label{ec38}
\end{eqnarray}
Again, one can use (\ref{ec38}) and the Borel-Cantelli Lemma to prove that \\ $\p{D_n > \delta\ \textrm{i.o.}}=0$ and then let $\delta\rightarrow 0$ through a decreasing sequence to obtain
\begin{equation}\label{ec61}
D_n\cas 0.
\end{equation}
Finally, one sees that
\[ \sup_{\psi\in\mathcal{A}_n} \left\{ \left|\frac{1}{n}\sum_{X_j\in\texttt{X}} \psi(X_j)(Y_j - \phi(X_j)) \right|\right\} = B_n \leq C_n + D_{\lfloor \sqrt{n} \rfloor},\]
which combined with (\ref{ec60}) and (\ref{ec61}) gives
\begin{equation}\nonumber
\lsup B_n \leq \eta\sqrt{M}\ \textrm{ almost surely.}
\end{equation}
Taking (\ref{ec36}) into account we get
\[ \lsup_{n\rightarrow\infty} \sup_{\psi\in\mathcal{D}_{K,\texttt{X}}} \left\{ \left|\frac{1}{n}\sum_{1\leq j \leq n: X_j\in\texttt{X}} \psi(X_j)(Y_j - \phi(X_j)) \right|\right\}\leq 2 \eta\sqrt{M}\ \textrm{ almost surely}.\]
Letting $\eta\rightarrow 0$ we get the desired result. $\hfill \square$

\subsection{Proof of Lemma \ref{l21}}\label{pl21}
We can assume, without loss of generality, that $\texttt{X}$ is a finite union of compact rectangles. Consider a sequence $(\texttt{X}_m)_{m=1}^\infty$ satisfying the following properties:
\begin{enumerate}[(a)]
\item $\texttt{X}\subset\texttt{X}_m\subset\mathfrak{X}^\circ$ $\forall$ $m\in\mathbb{N}$.
\item $\nu(\texttt{X}_m) > 1-\frac{1}{m}$ $\forall$ $m\in\mathbb{N}$.
\item $\texttt{X}_m\subset\texttt{X}_{m+1}$ $\forall$ $m\in\mathbb{N}$.
\item Every $\texttt{X}_m$ can be expressed as a finite union of compact rectangles with positive Lebesgue measure.
\end{enumerate}
The existence of such a sequence follows from the inner regularity of Borel probability measures on $\mathbb{R}^d$ and from the fact that since $\mathfrak{X}^\circ$ is open, for any compact set $F\subset\mathfrak{X}^\circ$ we can find a finite cover composed by compact rectangles with positive Lebesgue measure and completely contained in $\mathfrak{X}^\circ$. Also, from Lemmas \ref{l16}, \ref{l17} and \ref{l18} and the fact that $\mathfrak{X}\subset \textrm{Dom}(\phi)$, for any $m\in\mathbb{N}$ we can find $K_m > 0$ such that
\begin{eqnarray}
\|\phi\|_{\texttt{X}_m} \leq K_m & \mbox{ and } & \p{\|\hat{\phi}_n\|_{\texttt{X}_m} > K_m\ \textrm{ i.o.}} = 0; \label{eq:Star} \\
\sup_{\begin{subarray}{c} x\in\texttt{X}_m\\ \xi\in\partial\phi(x)\end{subarray}}\{|\xi|\} \leq K_m & \mbox{ and } & \op{\sup_{\begin{subarray}{c} x\in\texttt{X}_m \\ \xi\in\partial\hat{\phi}_n(x)\end{subarray}}\{|\xi|\} > K_m\ \textrm { i.o.}} = 0. \label{eq:StarStar}
\end{eqnarray}
Fix $\eta > 0$ and consider the sets
\begin{eqnarray}
A &=& \left[\inf_{x\in\texttt{X}}\{  \phi(x) - \hat{\phi}_n (x) \} \geq \eta\ \textrm{ i.o.}\right]\nonumber\\
B &=& \left[\|\hat{\phi}_n\|_{\texttt{X}_m} \leq K_m\ \textrm { a.a.}\right]\nonumber\\
C &=& \bigg[\sup_{\begin{subarray}{c} x\in\texttt{X}_m \\ \xi\in\partial\hat{\phi}_n(x)\end{subarray}}\{|\xi|\} \leq K_m\ \textrm { a.a.}\bigg].\nonumber
\end{eqnarray}
Suppose now that $A\cap B\cap C$ is known to be true. Then, there is a subsequence $(n_k)_{k=1}^\infty$ such that $\inf_{x\in\texttt{X}}\{  \phi(x) - \hat{\phi}_{n_k} (x) \} \geq \eta\ \forall \ k\in\mathbb{N}$ and $\frac{1}{n_k}\sum_{j=1}^{n_k} \e{\epsilon_j^2}\rightarrow \sigma^2$. Taking (\ref{eq:Star}) and (\ref{eq:StarStar}) into account, we have that for $k$ large enough the inequality
\begin{eqnarray*}
\frac{1}{n_k} \sum_{j=1}^{n_k} (Y_j - \hat{\phi}_{n_k}(X_j))^2 \geq \frac{1}{n_k} \sum_{X_j\in\texttt{X}_m} (Y_j - \phi(X_j))^2 \qquad \qquad \qquad\\
+ \frac{2}{n_k} \sum_{X_j\in\texttt{X}_m} (Y_j - \phi(X_j))(\phi(X_j) - \hat{\phi}_{n_k}(X_j)) + \frac{1}{n_k} \sum_{X_j\in\texttt{X}_m} (\phi(X_j) - \hat{\phi}_{n_k}(X_j))^2
\end{eqnarray*}
implies
\[\frac{1}{n_k} \sum_{j=1}^{n_k} (Y_j - \hat{\phi}_{n_k}(X_j))^2 \geq \frac{1}{n_k} \sum_{X_j\in\texttt{X}_m} (Y_j - \phi(X_j))^2\ +\]
\[ \frac{N_{n_k}(\texttt{X})}{n_k}\eta^2 - 4 \sup_{\psi\in\mathcal{D}_{K_m,\texttt{X}_m}} \left\{ \left|\frac{1}{n_k}\sum_{\{1\leq j\leq n_k: X_j\in\texttt{X}_m\}} \psi(X_j)(Y_j - \phi(X_j)) \right|\right\}.\]
Thus, from Lemma \ref{l20} we can conclude that
\begin{eqnarray}
\linf_{k\rightarrow\infty}\frac{1}{n_k} \sum_{1\leq j \leq n_k} (Y_j - \hat{\phi}_{n_k}(X_j))^2 \geq \nu(\texttt{X}_m)\sigma^2 + \nu(\texttt{X})\eta^2 \textrm{ if \{A1-A3\} hold.} \nonumber
\end{eqnarray}
Under \{A2-A4\} and \{A5-A7\} the left-hand side of the last display is bounded from below by
\[\linf_{k\rightarrow\infty}\frac{1}{n_k}\sum_{X_j\in\texttt{X}_m} (Y_j - \phi(X_j))^2 + \nu(\texttt{X})\eta^2\] and \[\int_{\texttt{X}_m} (y - \phi(x))^2\mu(dx,dy) + \nu(\texttt{X})\eta^2,\] respectively.

Finally, using (a)-(d), the strong law of large numbers (for \{A2-A4\} we can apply a version of the strong law of large numbers for independent random variables thanks to condition A4-(ii); see \cite{pwm}, Lemma 12.8, page 118 or \cite{fol}, Theorem 10.12, page 322) and Lemma \ref{lfdp} we can let $m\rightarrow\infty$ to see that, under any of \{A1-A3\}, \{A2-A4\} or \{A5-A7\},
\[ \linf_{k\rightarrow\infty}\frac{1}{n_k} \sum_{1\leq j \leq n_k} (Y_j - \hat{\phi}_{n_k}(X_j))^2 \geq \sigma^2 + \nu(\texttt{X})\eta^2\]
which is impossible because $\hat{\phi}_{n_k}$ is the least squares estimator.

Therefore $\op{A\cap B \cap C} = 0$ and, since $\ip{B\cap C}=1$, \[\p{A}=\p{\inf_{x\in\texttt{X}}\{  \phi(x) - \hat{\phi}_n (x) \} \geq \eta\ \textrm{ i.o.}}=0.\]
This finishes the proof of $(i)$. The second assertion follows from similar arguments. $\hfill \square$

\subsection{Proof of Lemma \ref{l23}}\label{pl23}
We can assume, without loss of generality, that $\texttt{X}$ is a finite union of compact rectangles. Pick $K_\texttt{X}$ such that
\begin{eqnarray}
\sup_{\begin{subarray}{c} x\in\texttt{X} \\ \xi\in\partial\phi(x)\end{subarray}}\{|\xi|\} \leq K_\texttt{X} \textrm{ and }\op{\sup_{\begin{subarray}{c} x\in\texttt{X}\\ \xi\in\partial\hat{\phi}_n(x)\end{subarray}}\{|\xi|\} > K_\texttt{X}\ \textrm { i.o.}} = 0.\nonumber
\end{eqnarray}
Let $\eta > 0$ and $\delta = \frac{\eta}{3K_\texttt{X}}$. We can then divide $\texttt{X}$ in $M$ subrectangles $\{\mathcal{C}_1,\ldots,\mathcal{C}_M\}$ all having diameter less than $\delta$. Define the events
\begin{eqnarray}
A &=& \left[\bigcap_{1\leq k \leq M} \inf_{x\in\mathcal{C}_k}\{  \hat{\phi}_n (x) - \phi(x) \} < \frac{\eta}{3}\ \textrm{ a.a.}\right]\nonumber\\
B &=& \bigg[\sup_{\begin{subarray}{c} x\in\texttt{X} \\ \xi\in\partial\hat{\phi}_n(x)\end{subarray}}\{|\xi|\} \leq K_\texttt{X}\ \textrm { a.a.}\bigg].\nonumber
\end{eqnarray}
We will show that $A\cap B \subset \left[\sup_{x\in\texttt{X}}\{ \hat{\phi}_n (x) - \phi(x) \} \leq \eta\ \textrm{ a.a.} \right]$. Suppose $A\cap B$ is true. Then, there is $N\in\mathbb{N}$ such that for any $n\geq N$ we can find $\Xi_{n,k}\in\mathcal{C}_k$ such that $\hat{\phi}_n (\Xi_{n,k}) - \phi(\Xi_{n,k}) < \frac{\eta}{3}$. Moreover, we can make $N$ large enough such that for any $n\geq N$, $K_\texttt{X}$ is an upper bound for all the subgradients of $\hat{\phi}_n$ on $\texttt{X}$. Then, for any $\xi\in\mathcal{C}_k$ we obtain from the Lipschitz property,
\begin{eqnarray}
\hat{\phi}_n(\xi) - \phi(\xi) &=&  (\hat{\phi}_n(\Xi_{n,k}) - \phi(\Xi_{n,k})) + (\phi(\Xi_{n,k})-\phi(\xi)) + (\hat{\phi}_n(\xi) - \hat{\phi}_n(\Xi_{n,k}))\nonumber\\
 &\leq& \frac{\eta}{3} + K_\texttt{X}\delta + K_\texttt{X}\delta\leq \eta.\nonumber
\end{eqnarray}
Therefore,
\[\sup_{x\in\mathcal{C}_k}\{  \hat{\phi}_n (x) - \phi(x) \} \leq \eta \ \ \ \forall\ 1\leq k \leq M\ \forall\ n\geq N\]
which implies
\[ \sup_{x\in\texttt{X}}\{  \hat{\phi}_n (x) - \phi(x) \} \leq \eta\ \ \ \forall\ n\geq N.\]
Considering Lemmas \ref{l21}-(ii) and \ref{l18}; $A\cap B \subset \left[\sup_{x\in\texttt{X}}\{ \hat{\phi}_n (x) - \phi(x) \} \leq \eta\ \textrm{ a.a.} \right]$ and $\ip{A\cap B} = 1$ we obtain $(ii)$. The first assertion follows from similar arguments and $(iii)$ is a direct consequence of $(i)$ and $(ii)$. $\hfill \square$

\subsection{Proof of Lemma \ref{l22}}\label{pl22}
Throughout this proof we will denote by $\textbf{B}$ the unit ball (w.r.t. the euclidian norm) in $\mathbb{R}^d$.
From Theorem 25.5, page 246 on \cite{rca} we know that $f$ is continuously differentiable on $\mathcal{C}$. Let
\[ h_* = \inf_{\begin{subarray}{c} \xi\in\texttt{X},\eta\in\mathbb{R}^d\setminus\mathcal{C}\end{subarray}} \{|\xi - \eta|\} > 0. \]
Pick $\epsilon > 0$. We will first show that there is $n_\epsilon \in \mathbb{N}$ such that
\begin{equation}\label{ec39}
\langle \xi , \eta \rangle \leq \langle \nabla f(x) , \eta \rangle + \epsilon,\ \ \ \forall\ \xi\in\partial f_n(x),\ \forall\ x\in\texttt{X},\ \forall \ \eta \in \textbf{B},\ \forall\ n\geq n_\epsilon.
\end{equation}
Suppose that such an $n_\epsilon$ does not exist. Then, there is an increasing sequence $(m_n)_{n=1}^\infty$ such that for any $n\in\mathbb{N}$ we can find $x_{m_n}\in\texttt{X}$, $\xi_{m_n}\in\partial f_{m_n}(x_{m_n})$, $\eta_{m_n}\in\textbf{B}$ satisfying $\langle \xi_{m_n} , \eta_{m_n} \rangle > \langle \nabla f(x_{m_n}) , \eta_{m_n} \rangle + \epsilon$. But $\texttt{X}$ and \textbf{B} are both compact, so there are $x_*\in\texttt{X}$, $\eta_*\in\textbf{B}$ and a subsequence $(k_n)_{n=1}^\infty$ of $(m_n)_{n=1}^\infty$ such that $x_{k_n}\rightarrow x_*$ and $\eta_{k_n}\rightarrow\eta_*$. Then, for any $0<h<h_*$ we have
\[ \frac{f_{k_n}(x_{k_n} + h \eta_{k_n})-f_{k_n}(x_{k_n})}{h} \geq \langle \xi_{k_n}, \eta_{k_n}\rangle > \langle \nabla f(x_{m_n}) , \eta_{k_n} \rangle + \epsilon\ \ \forall\ n\in\mathbb{N}, \]
and therefore
\[ \linf_{n\rightarrow\infty}\lim_{h\downarrow 0}\frac{f_{k_n}(x_{k_n} + h \eta_{k_n})-f_{k_n}(x_{k_n})}{h} \geq \langle \nabla f(x_*) , \eta_* \rangle + \epsilon.\]
But this is impossible in view of Theorem 24.5, page 233 on \cite{rca}. It follows that we can choose some $n_\epsilon\in\mathbb{N}$ with the property described in (\ref{ec39}). By noting that $-\textbf{B} = \textbf{B}$, we can conclude from (\ref{ec39}) that
\[ |\langle \xi , \eta \rangle - \langle \nabla f(x) , \eta \rangle| \leq \epsilon\ \ \ \forall\ \xi\in\partial f_n(x),\ \forall\ x\in\texttt{X},\ \forall\  \eta \ \in \textbf{B},\ \forall\ n\geq n_\epsilon.\]
By taking $\eta_\xi = \frac{\xi - \nabla f(x)}{|\xi - \nabla f(x)|}$ when $\xi\neq\nabla f(x)$ we get
\[ \sup_{\begin{subarray}{c}x\in\texttt{X}\\ \xi\in\partial f_n (x)\end{subarray}}\left\{ |\xi - \nabla f(x)|\right\} \leq \epsilon\ \ \ \forall\ n\geq n_\epsilon. \]
Since $\epsilon > 0$ was arbitrarily chosen, this completes the proof. $\hfill \square$

\appendix
\section{Appendix}\label{sec5}
\subsection{Results from convex analysis}
\begin{lemma}\label{l1}
Let $z\in\mathbb{R}^n$, $x_1,\ldots,x_n\in\mathbb{R}^d$ and define the function $g:\mathbb{R}^d\rightarrow\overline{\mathbb{R}}$ by
\[ g(x) = \inf\left\{ \sum_{k=1}^n \theta^k z^k : \sum_{k=1}^n \theta^k = 1,\ \sum_{k=1}^n \theta^k x_k = x,\ \theta\geq 0,\ \theta\in\mathbb{R}^n \right\}. \]
Then, $g$ defines a convex function whose effective domain is $\conv{x_1,\ldots,x_n}$. Moreover, if $\mathcal{K}_{x,z}$ is the collection of all proper convex functions $\psi$ such that $\psi(x_j)\leq z^j$ for all $j=1,\ldots,n$, then $g = \sup_{\psi\in\mathcal{K}_{x,z}}\{\psi\}. $
\end{lemma}
\begin{proof} To see that $g$ defines a convex function, for any $x\in\mathbb{R}^d$ write
\[ A_x = \left\{ \theta\in\mathbb{R}^n : \sum_{k=1}^n \theta^k = 1,\ \sum_{k=1}^n \theta^k x_k = x,\ \theta\geq 0 \right\}\]
and observe that for any $x,y\in\mathbb{R}^d $, $t\in (0,1)$, $\vartheta\in A_y$ and $\theta\in A_x$  we have $t\theta + (1-t)\vartheta \in A_{tx + (1-t)y} $ and hence
\[ \frac{g\left(tx + (1-t)y\right) - (1-t) \sum_{k=1}^n \vartheta^k z^k}{t} \leq \sum_{k=1}^n \theta^k z^k. \]
Taking infimum over $A_x$ and rearranging terms, we get
\[\frac{g\left(tx + (1-t)y\right) - t g(x)}{1-t} \leq \sum_{k=1}^n \vartheta^k z^k\]
and taking now the infimum over $A_y$ gives the desired convexity. The convention that $\inf(\emptyset) = +\infty$ shows that the effective domain is precisely the convex hull of $x_1,\ldots,x_n$. Finally, for any $\psi\in\mathcal{K}_{x,z}$ and $x\in\conv{x_1,\ldots,x_n}$ we have, for $\theta\in\mathbb{R}^n$ with $\theta\geq 0$, $x = \sum_{j=1}^n \theta^j x_j$ and $\sum_{j=1}^n \theta^j =1$,
\[ \psi(x)\leq \sum_{j=1}^n \theta^j \psi(x_j) \leq \sum_{j=1}^n \theta^j z^j\]
since $\psi(x_j)\leq z^j$ for any $j=1,\ldots,n$. The definition of $g$ as an infimum then implies that $\psi(x)\leq g(x)$ $\forall$ $\psi\in\mathcal{K}_{x,z}$, $x\in\conv{x_1,\ldots,x_n}$. The result then follows from the fact that $g\in\mathcal{K}_{x,z}$.
\end{proof}

\begin{lemma}\label{l8}
Let $z\in\mathbb{R}^n$, $x_1,\ldots,x_n\in\mathbb{R}^d$ and define the function $h:\mathbb{R}^d\rightarrow\overline{\mathbb{R}}$ by
\[ h(x) = \inf\left\{ \sum_{k=1}^n \theta^k z^k : \sum_{k=1}^n \theta^k = 1,\ \vartheta + \sum_{k=1}^n \theta^k X_k = x,\ \theta\geq 0,\ \theta\in\mathbb{R}^n, \vartheta\in\mathbb{R}^d_+ \right\} \]
Then, $h$ defines a convex, componentwise nonincreasing function whose effective domain is $\conv{x_1,\ldots,x_n}+\mathbb{R}^d_+$. Moreover, if $\mathcal{Q}_{x,z}$ is the collection of all componentwise nonincreasing, proper convex functions $\psi$ such that $\psi(x_j)\leq z^j$ for all $j=1,\ldots,n$, then $h = \sup_{\psi\in\mathcal{Q}_{x,z}}\{\psi\}.$
\end{lemma}
\begin{proof} The proof that $h$ is convex is similar to the proof that $g$ is convex in Lemma \ref{l1}. Now, if $x\leq y\in\mathbb{R}^d$, observe that for any $\theta\in\mathbb{R}^n$, $\vartheta\in\mathbb{R}^d_+$ with $\sum_{k=1}^n \theta^k = 1,\ \vartheta + \sum_{k=1}^n \theta^k X_k = x,\ \theta\geq 0$, we also have $\vartheta + (y-x) + \sum_{k=1}^n \theta^k X_k = y$ and $\vartheta + (y-x) \in\mathbb{R}_+^d$. Then, from the definition of $h$ we see that $h(x)\geq h(y)$. Thus, $h$ is componentwise nonincreasing. That the effective domain of $h$ is $\conv{x_1,\ldots,x_n}+\mathbb{R}^d_+$ is clear from the fact that for any $x$ not belonging to that set, the infimum defining $h(x)$ would be taken over the empty set. Finally, for any $\psi\in\mathcal{Q}_{x,z}$ and $x\in\conv{x_1,\ldots,x_n}+\mathbb{R}^d_+$ we have, for $\theta\in\mathbb{R}^n$ and $\vartheta\in\mathbb{R}^d_+$ with $\theta\geq 0$, $x = \vartheta + \sum_{j=1}^n \theta^j x_j$ and $\sum_{j=1}^n \theta^j =1$,
\[ \psi(x)\leq \function{\psi}{\sum_{j=1}^n \theta^j x_j} \leq \sum_{j=1}^n \theta^j \psi(x_j) \leq \sum_{j=1}^n \theta^j z^j\]
since $\psi(x_j)\leq z^j$ for any $j=1,\ldots,n$. The definition of $h$ as an infimum then implies that $\psi(x)\leq h(x)$ $\forall$ $\psi\in\mathcal{Q}_{x,z}$, $x\in\conv{x_1,\ldots,x_n}+\mathbb{R}^d_+$. The result then follows from the fact that $h\in\mathcal{Q}_{x,z}$.
\end{proof}

\subsection{Results from matrix algebra}\label{apapendice}
Before proving Lemma \ref{l12}, we need the following result.

\begin{lemma}\label{lemmaoptimization}
Let $j\in \{1,\ldots,d\}$, $\alpha\in\{-1,1\}^d$ and $\rho_* > 0$. Then, the optimal value of the optimization problem
\[ \begin{array}{ll}
\textrm{min} & \langle \alpha^j \textbf{e}_j, w_2 - w_1 \rangle\\
\textrm{s.t.} & \left| w_2 - \frac{3\rho_*}{8\sqrt{d}}\alpha \right| \leq \frac{\rho_*}{8\sqrt{d}}\\
 & |w_1| \leq \frac{\rho_*}{16\sqrt{d}}\\
 & w_1,w_2\in\mathbb{R}^d
\end{array}
\]
is $\frac{3}{16\sqrt{d}}\rho_*$ and it is attained at $w_1^* = \frac{\rho_*}{16\sqrt{d}}\alpha^j\textbf{e}_j$ and $w_2^* = \frac{3\rho_*}{8\sqrt{d}}\alpha - \frac{\rho_*}{8\sqrt{d}}\alpha^j\textbf{e}_j$.
\end{lemma}
\begin{proof} Writing $w=(w_1;w_2)$ with $w_1,w_2\in\mathbb{R}^d$ for any $w\in\mathbb{R}^{2d}$, consider $f,g_1,g_2:\mathbb{R}^{2d}\rightarrow\mathbb{R}$ defined as:
\begin{eqnarray}
f(w) &=& \langle \alpha^j \textbf{e}_j, w_2 - w_1 \rangle, \nonumber \\
g_1(w) &=&  \frac{1}{2}\left(\left(\frac{\rho_*}{16\sqrt{d}}\right)^2 - |w_1|^2\right), \nonumber \\
g_2(w) &=&  \frac{1}{2}\left(\left(\frac{\rho_*}{8\sqrt{d}}\right)^2 - \left| w_2 - \frac{3\rho_*}{8\sqrt{d}}\alpha \right|^2\right). \nonumber
\end{eqnarray}
Then, $f,g_1,g_2$ are twice continuously differentiable on $\mathbb{R}^{2d}$ and the optimization problem can be re-written as minimizing $f(w)$ over the set $\{w\in\mathbb{R}^{2d} : g_1(w)\geq 0, g_2(w)\geq 0\}$. The proof now follows by noting that the vector $w^* = (w_1^*;w_2^*)\in\mathbb{R}^{2d}$ and the Lagrange multipliers $\lambda_1^* = \frac{16\sqrt{d}}{\rho_*}$ and $\lambda_2^* = \frac{8\sqrt{d}}{\rho_*}$ are the only ones which satisfy the Karush-Kuhn-Tucker second order necessary and sufficient conditions for a strict local solution to this problem as stated in Theorem 12.5, page 343 and Theorem 12.6, page 345 in \cite{nowr}.
\end{proof}

\subsubsection{Proof of Lemma \ref{l12}}\label{pl12}
Without loss of generality, we may assume that $r=1$. Let $R_r$ be $\frac{1}{\sqrt{d}}$ and pick $\delta \in \left(0,\frac{1}{\sqrt{d}}\right)$, $\rho_* = \frac{1}{\sqrt{d}} - \delta$ and $\rho^* = \frac{2d}{1-\delta\sqrt{d}}$. Consider a matrix $Z = (z_1,\ldots,z_d)\in\mathbb{R}^{d\times d}$ with columns $z_1,\ldots,z_d\in\mathbb{R}^d$ and define the function $\tilde{\xi}:\mathbb{R}^{d\times d}\rightarrow\mathbb{R}^d$ as
\begin{equation}\nonumber
\tilde{\xi}(Z) = \left|\begin{array}{cccc}
\textbf{e}_1 & z_2^1 - z_1^1 & \cdots & z_d^1 - z_1^1 \\
\vdots & \vdots & \vdots & \vdots \\
\textbf{e}_d & z_2^d - z_1^d & \cdots & z_d^d - z_1^d
\end{array}\right|
\end{equation}
where the bars denote the determinant and the equation is written symbolically to express that $\tilde{\xi}(Z)$ is a linear combination of the vectors $\{\textbf{e}_j\}_{1\leq j\leq d}$ with the cofactor corresponding to the $(j,1)$-th position as the coefficient of $\textbf{e}_j$. This is a common notation for ``generalized vector products''; see, for instance, \cite{caj}, Section 2.4.b, page 187 for more details. Since the determinant and all cofactors can be seen as a continuous function on $\mathbb{R}^{d\times d}$, it follows that $\tilde{\xi}$ is continuous on $\mathbb{R}^{d\times d}$. Now choose $\alpha\in\{-1,1\}^d$ and observe that
\begin{eqnarray}
\tilde{\xi}(\alpha^1\textbf{e}_1,\ldots,\alpha^d\textbf{e}_d) &=& \left(\prod_{j=1}^d \alpha^j\right)\alpha, \nonumber\\
\left|\tilde{\xi}(\alpha^1\textbf{e}_1,\ldots,\alpha^d\textbf{e}_d)\right| &=& \sqrt{d}, \nonumber\\
\langle \tilde{\xi}(\alpha^1\textbf{e}_1,\ldots,\alpha^d\textbf{e}_d), \alpha^j\textbf{e}_j \rangle &=& \prod_{k=1}^d \alpha^k\ \ \forall\ j=1,\ldots,d.\nonumber
\end{eqnarray}
Since $\mathbb{R}^{d\times d}$ has the product topology of the $d$-fold topological product of $\mathbb{R}^d$ with itself, the continuity of $\tilde{\xi}$ and of $\langle\cdot,\cdot\rangle$ imply that we can find $\rho_\alpha\in \left(0, \frac{1}{\sqrt{d}}-\delta\right)$ such that if $x_j\in B(\alpha^j\textbf{e}_j,\rho_\alpha)$ for any $j=1,\ldots,d$, $\beta = \{x_1,\ldots,x_d\}$ and $X_\beta = (x_1,\ldots,x_d)$, then
\begin{eqnarray}
\left|\left|\tilde{\xi}(X_\beta)\right|-\sqrt{d}\right| &<& \delta, \nonumber\\
\left|\frac{\tilde{\xi}(X_\beta)}{|\tilde{\xi}(X_\beta)|} - \frac{\prod_{1\leq j\leq d}\alpha^j}{\sqrt{d}}\alpha\right| &<& \delta, \label{ec10}\\
\left| \left \langle \frac{\tilde{\xi}(X_\beta)}{|\tilde{\xi}(X_\beta)|},x_j \right\rangle - \frac{\prod_{k = 1}^d \alpha^k}{\sqrt{d}}\right| &<& \delta\ \ \forall\ j=1,\ldots,d\label{ec11}.
\end{eqnarray}
Taking this into account, define
\begin{eqnarray}
\xi_{\alpha,\beta} = \left(\prod_{j=1}^d \alpha^j\right) \frac{\tilde{\xi}(X_\beta)}{|\tilde{\xi}(X_\beta)|}, \textrm{ and } b_{\alpha,\beta} = \langle \xi_{\alpha,\beta}, x_1 \rangle \nonumber.
\end{eqnarray}
From the definition of the function $\tilde{\xi}$ it is straight forward to see that $\langle \xi_{\alpha,\beta}, x_j - x_1\rangle = 0$ $\forall j\in\{1,\ldots,d\}$, so we in fact have
\[ x_1,\ldots,x_d\in\mathcal{H}_{\alpha,\beta}:= \{x\in\mathbb{R}^d: \langle \xi_{\alpha,\beta},x\rangle = b_{\alpha,\beta} \}. \]
Moreover, (\ref{ec10}) and (\ref{ec11}) imply
\begin{eqnarray}
\frac{1}{\sqrt{d}}+\delta > b_{\alpha,\beta} > \frac{1}{\sqrt{d}}-\delta > 0, \nonumber\\
\min_{1\leq j \leq d} \left\{ |\xi_{\alpha,\beta}^j | \right\} > \frac{1}{\sqrt{d}}-\delta>0. \nonumber
\end{eqnarray}
For simplicity, and without loss of generality (the other cases follow from symmetry), we now assume that $\alpha = \textbf{e}$, the vector of ones. By solving the corresponding quadratic programming problems, it is not difficult to see that
\begin{eqnarray}
\rho_* = \frac{1}{\sqrt{d}} - \delta < b_{\alpha,\beta} &=& \inf_{\langle \xi_{\alpha,\beta},x\rangle \geq b_{\alpha,\beta}} \{ |x| \} \nonumber\\
\rho^* = \frac{2d}{1-\delta\sqrt{d}} > \frac{b_{\alpha,\beta}}{\min_{1\leq j \leq d} \{ |\xi_{\alpha,\beta}^j | \}} &=&
\sup_{\begin{subarray}{c} \langle \xi_{\alpha,\beta},x\rangle \leq b_{\alpha,\beta} \\ x\geq 0\end{subarray}} \{ |x| \}.\nonumber
\end{eqnarray}
For the first inequality see, for instance, Exercise 16.2, page 484 of \cite{nowr}. For the second one, one must notice that $2\sqrt{d} > \frac{1}{\sqrt{d}}+\delta > b_{\alpha,\beta}$ and that the optimal value of the optimization problem must be attained at one of the vertices of the polytope $\{x\in\mathbb{R}^d_+:\langle \xi_{\alpha,\beta},x\rangle \leq b_{\alpha,\beta}\}$. The latter statement can be derived from the Karush-Kuhn-Tucker conditions of the problem.

The inequalities in the last display imply that $B(0,\rho_*)\subset\mathcal{H}_{\alpha,\beta}^-$ and
$\{x\in\mathbb{R}^d: |x| \geq \rho^*\}\cap\mathcal{R}_\alpha \subset \mathcal{H}_{\alpha,\beta}^+$.

Finally, for $x \in B(- \alpha^j \textbf{e}_j, \frac{1}{2}\rho_\alpha)$ we have $|x+x_j|<\rho_\alpha$ and therefore $\langle \xi_{\alpha,\beta},x\rangle < -\langle \xi_{\alpha,\beta},x_j\rangle + \rho_\alpha < \delta - \frac{1}{\sqrt{d}} + \rho_\alpha < 0$. We can then take any $\rho \leq \frac{1}{2}\min_{\alpha \in\{-1,1\}^d} \{\rho_\alpha\}$ to make $(i)$-$(vi)$ be true. We'll now argue that by making $\rho$ smaller, if required, $(vii)$ also holds.

Let $B_1 = B\left(0,\frac{\rho_*}{16\sqrt{d}}\right)$, $B_2 = B\left(\frac{3\rho_*}{8\sqrt{d}}\alpha,\frac{\rho_*}{8\sqrt{d}}\right)$ and consider the functions $\varphi,\psi:\mathbb{R}^{d\times d}\rightarrow\mathbb{R}$ given by
\begin{eqnarray}
\varphi(X) &=& \inf_{w_1\in B_1, w_2\in B_2}\left\{ \min_{1\leq j\leq d} \left\{ \left(X(w_2 - w_1)\right)^j\right\}\right\},\nonumber \\
\psi(X) &=& \sup_{w_1\in B_1}\left\{ \max_{1\leq j\leq d} \left\{ \left(X w_1\right)^j\right\}\right\}.\nonumber
\end{eqnarray}
Both of these functions are Lipschitz continuous with the metric induced by the $\|\cdot\|_2$-norm on $\mathbb{R}^{d\times d}$ with Lipschitz constants smaller than $\rho_*$. To see this, observe that $$|X(w_2-w_1) - Y(w_2-w_1)|\leq \|X-Y\|_2 |w_2 - w_1| \leq \frac{9}{16}\rho_* \|X-Y\|_2$$ for all $w_1\in B_1$, $w_2\in B_2$ and $X,Y\in\mathbb{R}^{d\times d}$. Also, simple algebra shows that $\left|\min_{1\leq j \leq d} \{x^j\} - \min_{1\leq j \leq d} \{y^j\}\right|\leq |x-y|$ $\forall$ $x,y\in\mathbb{R}^d$. From these assertions, one immediately gets the Lipschitz continuity of $\varphi$. Similar arguments show the same for $\psi$.

Let $\mathcal{I}_\alpha\in\mathbb{R}^{d\times d}$ be the diagonal matrix whose $j$'th diagonal element is precisely $\alpha^j$. From Lemma \ref{lemmaoptimization} it is seen that $\varphi(\mathcal{I}_\alpha) = \frac{3\rho_*}{16\sqrt{d}}$. On the other hand, it is immediately obvious that $\psi(\mathcal{I}_\alpha) = \frac{\rho_*}{16\sqrt{d}}$. Using one more time the continuity of $\psi$ and $\varphi$ and that the topology in $\mathbb{R}^{d\times d}$ is the same as the topology of the $d$-fold topological product of $\mathbb{R}^d$, for each $\alpha\in\{-1,1\}^d$ we can find $r_\alpha$ for which $X_\beta = (x_1,\ldots,x_d)\in\mathbb{R}^{d\times d}$ and $|x_j-\alpha^j \textbf{e}_j|<r_\alpha$ for all $j=1,\ldots,d$ imply $|\psi(X_\beta^{-1}) - \frac{\rho_*}{16\sqrt{d}}| < \frac{\rho_*}{32\sqrt{d}}$ and $|\varphi(X_\beta^{-1}) - \frac{3\rho_*}{16\sqrt{d}}| < \frac{\rho_*}{16\sqrt{d}}$. It follows that
\begin{eqnarray}
&& \inf_{\begin{subarray}{c} t\geq 1 \\ w_1 \in B_1, w_2 \in B_2\end{subarray}}\left\{ \min_{ 1\leq j\leq d}\left\{\left(X_\beta^{-1}(w_1 + t(w_2 - w_1))\right)^j\right\}\right\} \nonumber \\
& \ge & \inf_{\begin{subarray}{c} t\geq 1 \\ w_1 \in B_1, w_2 \in B_2\end{subarray}}\left\{ \min_{ 1\leq j\leq d}\left\{\left(t X_\beta^{-1}(w_2 - w_1)\right)^j\right\}\right\} - \sup_{w_1 \in B_1}\left\{ \max_{ 1\leq j\leq d}\left\{\left(X_\beta^{-1}w_1 \right)^j\right\}\right\} \nonumber \\
&\geq & \varphi(X_\beta^{-1}) - \psi(X_\beta^{-1}) \; \; > \;\; \frac{\rho_*}{8\sqrt{d}} - \frac{3\rho_*}{32\sqrt{d}} \; \; = \;\; \frac{\rho_*}{32\sqrt{d}} > 0.\nonumber
\end{eqnarray}
The proof is then finished by taking $\rho\leq\min_{\alpha\in\{-1,1\}^d}\left\{r_\alpha\land\frac{\rho_\alpha}{2}\right\}.$ $\hfill \square$

\subsubsection{Proof of Lemma \ref{l13}}\label{pl13}
Assume again, without loss of generality, that $r=1$. Lemma \ref{l12} $(ii)$ and $(vi)$ imply that $x_{\alpha^j j},x_{-\alpha^j j}\in\{x\in\mathbb{R}^d : \langle x, \xi_{\alpha}\rangle \leq b_\alpha\}$ for any $j=1,\ldots,n$ and any $\alpha\in\{-1,1\}^d$. It follows that, in addition to being convex, $\cap_{\alpha\in\{-1,1\}^d} \{x\in\mathbb{R}^d: \langle \xi_\alpha , x \rangle \leq b_\alpha \}$ contains $\{x_{\pm 1},\ldots, x_{\pm d}\}$ and hence it must contain $K$. For the other contention, take $x\in\cap_{\alpha\in\{-1,1\}^d} \{w\in\mathbb{R}^d: \langle \xi_\alpha , w \rangle \leq b_\alpha \}$ with $x\neq 0$ and any $\alpha\in\{-1,1\}^d$ for which $x\in\mathcal{R}_\alpha$. Then, $\langle \xi_{\alpha},x \rangle > 0$ for otherwise we would have
\[ \kappa x\in \mathcal{R}_\alpha\setminus \mathcal{H}_\alpha^+\ \ \forall\ \kappa \geq 0 \]
which is impossible by $(v)$ in Lemma \ref{l12}. Thus, $\mathcal{J}_x = \{\alpha\in\{-1,1\}^d : \langle \xi_\alpha,x \rangle > 0\} \neq \emptyset$ and we can define
\begin{eqnarray}
r_x = \min_{\alpha\in\mathcal{J}_x} \left\{ \frac{b_\alpha}{\langle \xi_\alpha,x \rangle}\right\} \ \mbox{ and } \alpha_x = \argmin_{\alpha\in\mathcal{J}_x} \left\{ \frac{b_\alpha}{\langle \xi_\alpha,x \rangle}\right\}.\nonumber
\end{eqnarray}
Note that $r_x\geq 1$. Since $\beta_{\alpha_x}$ is a basis, there is $\theta\in\mathbb{R}^d$ such that $r_x x = \theta^1 x_{\alpha_x^1 1} + \ldots + \theta^d x_{\alpha_x^d d}$. But then, $$b_{\alpha_x} = \langle r_x x,\xi_{\alpha_x}\rangle = \sum_{k=1}^d \theta^k \langle x_{\alpha^k_x k},\xi_{\alpha_x}\rangle = b_{\alpha_x} \sum_{k=1}^d \theta^k $$ where the last equality follows from $(ii)$ of Lemma \ref{l12} and therefore $\theta^1 + \ldots + \theta^d = 1$. Now assume that $\theta^j < 0$ for some $j\in\{1,\ldots,d\}$ and set $\gamma_x\in\{-1,1\}^d$ with $\gamma_x^k = \alpha_x^k$ for $k\neq j$ and $\gamma_x^j = - \alpha_x^j$.
But then, $\sum_{k\neq j} \theta^k = 1 - \theta^j >1$, $\langle x_{\alpha_x^k k}, \xi_{\gamma_x}\rangle = b_{\gamma_x}$ for $k\neq j$ and $\langle x_{\alpha_j^j j}, \xi_{\gamma_x}\rangle < 0$ by $(ii)$ and $(vi)$ in Lemma \ref{l12}. Therefore,
\begin{eqnarray}
\langle r_x x ,\xi_{\gamma_x} \rangle &=& \theta^j \langle x_{-\alpha_x^j j}, \xi_{\gamma_x}\rangle + \sum_{k\neq j} \theta^k \langle x_{\alpha_x^k k}, \xi_{\gamma_x}\rangle\label{ec16}\\
 &>& \sum_{k\neq j} \theta^k \langle x_{\alpha_x^k k}, \xi_{\gamma_x}\rangle > b_{\gamma_x}\label{ec17}
\end{eqnarray}
which is impossible because it contradicts the definition of $r_x$. Hence, $\theta\geq 0$ and we have $r_x x\in \conv{\beta_{\alpha_x}}$. Note that since 0 belongs in the interior of $\cap_{\alpha\in\{-1,1\}^d} \{w\in\mathbb{R}^d: \langle \xi_\alpha , w \rangle \leq b_\alpha \}$, there there is $\kappa>0$ such that $-\kappa x\in\cap_{\alpha\in\{-1,1\}^d} \{w\in\mathbb{R}^d: \langle \xi_\alpha , w \rangle \leq b_\alpha \}$. Applying the same arguments as before to $-\kappa x$ instead of $x$, we can find $\tilde{r}_x >0$ and $\tilde{\alpha}_x\in\{-1,1\}^d$ such that $-\tilde{r}_x x \in \conv{\beta_{\tilde{\alpha}_x}}$. It follows that $- \tilde{r}_x x, r_x x \in K$ and therefore $0,x\in K$ since $r_x \geq 1$. Hence, we have proved $(i)$.

To prove $(ii)$, note that $A:=\cap_{\alpha\in\{-1,1\}^d} \{w\in\mathbb{R}^d: \langle \xi_\alpha , w \rangle < b_\alpha \}$ is open and, by $(i)$, it is contained in $K$. Thus, $A\subset K^\circ$. That $K^\circ\subset A$ follows from the fact that if $x\in K\setminus A$, then $\langle \xi_\alpha ,x \rangle = b_\alpha$ for some $\alpha\in\{-1,1\}^d$, which implies that $B(x,\tau)\cap\textrm{Ext}(K)\neq \emptyset$ for all $\tau>0$ and hence $x\notin K^\circ$.

It is then obvious that $(iv)$ follows from the identity $\partial K = \overline{K} \setminus K^\circ$ and the fact that $K$ is closed.

Pick any $\alpha\in\{-1,1\}^d$ and observe that $(ii)$ and $(vi)$ from Lemma \ref{l12} imply that for any $\gamma\in\{-1,1\}^d$ we have
\[ \langle \xi_\gamma, x_{\alpha^k k} \rangle \left\{ \begin{array}{cl}
 = b_\gamma & \textrm{ if } \gamma^k =\alpha^k\\
 <0 \leq b_\gamma & \textrm{ if } \gamma^k = - \alpha^k
 \end{array}\right.
\]
which by $(iv)$ of this lemma show that
\[x_{\alpha^j j}\in  \{w\in\mathbb{R}^d: \langle \xi_\alpha , w \rangle = b_\alpha \}\cap \left( \cap_{\gamma\in\{-1,1\}^d} \{w\in\mathbb{R}^d: \langle \xi_\gamma , w \rangle \leq b_\gamma \}\right)\]
for all $\alpha\in\{-1,1\}^d$ and $j=1,\ldots, d$. Since the sets on the right-hand side of the last display are all convex we can conclude that
{\small \[ \conv{x_{\alpha^1 1},\ldots,x_{\alpha^j j}}\subset \{w\in\mathbb{R}^d: \langle \xi_\alpha , w \rangle = b_\alpha \}\cap\left( \cap_{\gamma\in\{-1,1\}^d} \{w\in\mathbb{R}^d: \langle \xi_\gamma , w \rangle \leq b_\gamma \}\right)\]}
for all $\alpha\in\{-1,1\}^d$. Thus, $ \bigcup_{\alpha\in\{-1,1\}^d} \conv{x_{\alpha^1 1},\ldots,x_{\alpha^j j}} \subset \partial K$.
Finally, take $x\in\partial K$. Then, there is $\alpha_x\in\{-1,1\}^d$ such that $\langle \xi_{\alpha_x}, x \rangle = b_{\alpha_x}$. Since $\beta_{\alpha_x}$ is a basis we can again find $\theta\in\mathbb{R}^d$ such that $x = \theta^1 x_{\alpha_x^1 1} + \ldots + \theta^d x_{\alpha_x^d d}$. Just as before, $\langle \xi_{\alpha_x} , x_{\alpha_x^j j} \rangle = b_{\alpha_x}$ implies that $\sum \theta^j = 1$. And again, if $\theta^j < 0$ for some $j$, we can take  $\gamma_x\in\{-1,1\}^d$ with $\gamma_x^k = \alpha_x^k$ for $k\neq j$ and $\gamma_x^j = - \alpha_x^j$ and arrive at a contradiction with similar arguments to those used in (\ref{ec16}) and (\ref{ec17}). This shows that $x\in\conv{\beta_{\alpha_x}}$ and completes the proof as $(v)$ and $(vi)$ are direct consequences of $(i)-(iv)$ and Lemma \ref{l12}. $\hfill \square$

\subsubsection{Proof of Lemma \ref{l14}}\label{pl14}
Let $r\in(0,\frac{1}{d-2})$ if $d\ge 3$ and $r>0$ if $d\le 2$. Since the geometric properties of any rectangle depend only on the direction and magnitude of the diagonal, we may assume without loss of generality that $b>0$ and that $a=\frac{r}{1+r} b$. This is because we can define $\tilde{b} = (1+r)(b-a) > 0$ and $\tilde{a} = a-r(b-a)$ to obtain $[a,b] = \tilde{a} + \left[\frac{r}{r+1}\tilde{b},\tilde{b}\right]$. For any $\alpha\in\{-1,1\}^d$, define $\alpha_j = \alpha - 2\alpha^j \textbf{e}_j \in \mathbb{R}^d$ and $w_\alpha = z_\alpha + r(z_\alpha - z_{-\alpha})$. Additionally, define the functions $\psi_\alpha,\varphi_\alpha :\mathbb{R}^{d\times d}\times\mathbb{R}^d\rightarrow\mathbb{R}$ by
\begin{eqnarray}
\psi_\alpha (\Theta,\theta) &=& \langle \textbf{e}, \Theta(z_\alpha - \theta) \rangle\nonumber\\
\varphi_\alpha(\Theta,\theta) &=& \min_{1\leq j \leq d} \left\{\left(\Theta(z_\alpha - \theta)\right)^j\right\}.\nonumber
\end{eqnarray}
Considering $\mathbb{R}^{d\times d}$ with the topology generated be the $\|\cdot\|_2$ norm and $\mathbb{R}^{d\times d}\times \mathbb{R}^d$ with the product topology, it is easily seen that both functions defined in the last display are continuous. Now, let $W_\alpha\in\mathbb{R}^{d\times d}$ be the matrix whose $j$'th column is precisely $w_{\alpha_j} - w_\alpha$. It is not difficult to see that $\psi_\alpha (W_\alpha^{-1},w_\alpha) = \frac{dr}{1+2r} < 1$ and $\varphi_\alpha (W_\alpha^{-1},w_\alpha) = \frac{r}{1+2r} > 0$. For instance, one can check that for $\alpha = -\textbf{e}$, one has $w_\alpha = 0$ and $w_{\alpha_j} = \frac{1+2r}{1+r}b^j\textbf{e}_j$ and the result is now evident. By symmetry, the same is true for any $\alpha\in\{-1,1\}^d$. Therefore, for any $\alpha\in\{-1,1\}^d$ there is $\rho_\alpha$ such that whenever $|x_{\alpha_j} - w_{\alpha_j}|<\rho_\alpha$ $\forall$ $j=1,\ldots,d$ and $X_\alpha$ is the matrix whose $j$'th column is $x_{\alpha_j} - x_\alpha$, we get
\begin{eqnarray}
\psi_\alpha (X_\alpha^{-1},x_\alpha) &<& 1, \label{ec18}\\
\varphi_\alpha (X_\alpha^{-1},x_\alpha) &>& 0.\label{ec19}
\end{eqnarray}
Letting $\rho = \min_{\alpha\in\{-1,1\}^d}\left\{\rho_\alpha\right\}$ completes the proof as (\ref{ec18}) and (\ref{ec19}) imply $z_\alpha\in\conv{x_\alpha,x_{\alpha_1}, \ldots,x_{\alpha_d}}^\circ$. $\hfill \square$

\bibliography{Referencias}

\begin{thebibliography}{}

\bibitem[Allon et~al., 2007]{necpt}
Allon, G., Beenstock, M., Hackman, S., Passy, U., and Shapiro, A. (2007).
\newblock Nonparametric estimation of concave production technologies by
  entropic methods.
\newblock {\em J. Appl. Econometrics}, 4:795--816.

\bibitem[Banker and Maindiratta, 1992]{bama}
Banker, R. and Maindiratta, A. (1992).
\newblock Maximum likelihood estimation of monotone and concave production
  frontiers.
\newblock {\em J. Productiv. Anal.}, 3:401--415.

\bibitem[Beresteanu, 2007]{bersieve}
Beresteanu, A. (2007).
\newblock Nonparametric estimation of regression functions under restrictions
  on partial derivatives.
\newblock Available at: http://econ.duke.edu/~arie/shape.pdf (unpublished
  manuscript).

\bibitem[Birke and Dette, 2006]{biho}
Birke, M. and Dette, H. (2006).
\newblock Estimating a convex function in nonparametric regression.
\newblock {\em Scand. J. Statist.}, 34:384--404.

\bibitem[Boland, 1997]{bolqp}
Boland, N.~L. (1997).
\newblock A dual-active-set algorithm for positive semi-definite quadratic
  programming.
\newblock {\em Math. Program.}, 78:1--27.

\bibitem[Bron\u{s}te\u{\i}n, 1978]{bro78a}
Bron\u{s}te\u{\i}n, E.~M. (1978).
\newblock Extremal convex functions.
\newblock {\em Sibirsk. Mat. Zh.}, 19:10--18.

\bibitem[Brunk, 1955]{bru}
Brunk, H.~D. (1955).
\newblock Maximum likelihood estimates of monotone parameters.
\newblock {\em Ann. Math. Statist.}, 26:607--616.

\bibitem[Brunk, 1970]{bru70}
Brunk, H.~D. (1970).
\newblock Estimation of isotonic regression.
\newblock In {\em Nonparametric Techniques in Statistical Inference}, pages
  177--197. Camridge University Press, New York, NY, USA.

\bibitem[Chung, 2001]{chung}
Chung, K.~L. (2001).
\newblock {\em A Course in Probability Theory}.
\newblock Academic Press, San Diego, CA, USA.

\bibitem[Conway, 1985]{con}
Conway, J. (1985).
\newblock {\em A Course in Functional Analysis}.
\newblock Springer-Verlag, New York, NY, USA.

\bibitem[Courant and John, 1999]{caj}
Courant, R. and John, F. (1999).
\newblock {\em Introduction to Calculus and Analysis, Vol. II/1}.
\newblock Springer, New York, NY, USA.

\bibitem[Cule and Samworth, 2010]{cusa09}
Cule, M. and Samworth, R. (2010).
\newblock Theoretical properties of the log-concave maximum likelihood
  estimator of a multidimensional density.
\newblock {\em Electron. J. Stat.}, 4:254--270.

\bibitem[Cule et~al., 2010]{cusas08}
Cule, M., Samworth, R., and Stewart, M. (2010).
\newblock Maximum likelihood estimation of a multidimensional log-concave
  density.
\newblock {\em J. R. Stat. Soc. Ser. B (to appear)}.

\bibitem[Dudley, 1977]{du77}
Dudley, R.~M. (1977).
\newblock On second derivatives of convex functions.
\newblock {\em Math. Scand.}, 41:159--174.

\bibitem[Folland, 1999]{fol}
Folland, G. (1999).
\newblock {\em Real Analysis: Modern Techniques and Their Applications}.
\newblock John Wiley \& Sons, New York, NY, USA.

\bibitem[Grenander, 1956]{gre}
Grenander, U. (1956).
\newblock On the theory of mortality measurement.
\newblock {\em Skan. Aktuarietidskr, Part II}, 39:125--153.

\bibitem[Groeneboom et~al., 2001]{gjw}
Groeneboom, P., Jongbloed, G., and Wellner, J. (2001).
\newblock Estimation of a convex function: characterizations and asymptotic
  theory.
\newblock {\em Ann. Statist.}, 29:1653--1698.

\bibitem[Hanson and Pledger, 1976]{hp76}
Hanson, D.~L. and Pledger, G. (1976).
\newblock Consistency in concave regression.
\newblock {\em Ann. Statist.}, 4:1038--1050.

\bibitem[Harville, 2008]{masp}
Harville, D. (2008).
\newblock {\em Matrix Algebra from a Statistician's Perspective}.
\newblock Springer, New York, NY, USA.

\bibitem[Hildreth, 1954]{hil}
Hildreth, C. (1954).
\newblock Estimates of ordinates of concave functions.
\newblock {\em J. Amer. Statist. Assoc.}, 49:598--619.

\bibitem[Johansen, 1974]{johan74}
Johansen, S. (1974).
\newblock The extremal convex functions.
\newblock {\em Math. Scand.}, 41:61--68.

\bibitem[Kapoor and Vaidya, 1986]{kavaqp}
Kapoor, S. and Vaidya, P.~M. (1986).
\newblock Fast algorithms for convex quadratic programming and multicommodity
  flows.
\newblock {\em Proceedings of the eighteenth annual ACM symposium on Theory of
  computing}, pages 147--159.

\bibitem[Kuosmanen, 2008]{kuo}
Kuosmanen, T. (2008).
\newblock Representation theorem for convex nonparametric least squares.
\newblock {\em Econom. J.}, 11:308--325.

\bibitem[Luenberger, 1984]{lbg}
Luenberger, D. (1984).
\newblock {\em Linear and Nonlinear Programming}.
\newblock Addison-Wesley Publishing Company, Reading, MA, USA.

\bibitem[Maindiratta and Sarath, 1997]{masa}
Maindiratta, A. and Sarath, B. (1997).
\newblock On the consistency of maximum likelihood estimation of monotone and
  concave production frontiers.
\newblock {\em J. Productiv. Anal.}, 8:239--246.

\bibitem[Mammen, 1991]{mam}
Mammen, E. (1991).
\newblock Nonparametric regression under qualitative smoothness assumptions.
\newblock {\em Ann. Statist.}, 19:741--759.

\bibitem[Matzkin, 1991]{mat1}
Matzkin, R.~L. (1991).
\newblock Semiparametric estimation of monotone concave utility functions for
  polychotomous choice models.
\newblock {\em Econometrica}, 59:1351--1327.

\bibitem[Matzkin, 1993]{mat2}
Matzkin, R.~L. (1993).
\newblock Nonparametric identification and estimation of polychotomous choice
  models.
\newblock {\em J. Econometrics}, 58:137--168.

\bibitem[Mehrotra and Sun, 1990]{mesuqp}
Mehrotra, S. and Sun, J. (1990).
\newblock An algorithm for convex quadratic programming that requires
  $o(n^{3.5}l)$ arithmetic operations.
\newblock {\em Math. Oper. Res.}, 15:342--363.

\bibitem[Moreau, 1962]{mor}
Moreau, J. (1962).
\newblock Decomposition orthogonal d'un espace hilbertien selon deux cones
  mutuellement polaires.
\newblock {\em C. R. Acad. Sci. Paris}, 255:238--240.

\bibitem[Nocedal and Wright, 1999]{nowr}
Nocedal, J. and Wright, S. (1999).
\newblock {\em Numerical Optimization}.
\newblock Springer, New York, NY, USA.

\bibitem[Rockafellar, 1970]{rca}
Rockafellar, T.~R. (1970).
\newblock {\em Convex Analysis}.
\newblock Princeton University Press, Princeton, NJ, USA.

\bibitem[Schuhmacher and D{\"u}mbgen, 2010]{scdu}
Schuhmacher, D. and D{\"u}mbgen, L. (2010).
\newblock Consistency of multivariate log-concave density estimators.
\newblock {\em Statist. Probab. Lett.}, 80:376--380.

\bibitem[Schuhmacher et~al., 2009]{schudu}
Schuhmacher, D., H{\"u}sler, A., and D{\"u}mbgen, L. (2009).
\newblock Multivariate log-concave distributions as a nearly parametric model.
\newblock Tech. rep., University of Bern. Available at:
  http://arxiv.org/abs/0907.0250.

\bibitem[Seregin and Wellner, 2009]{sewe}
Seregin, A. and Wellner, J. (2009).
\newblock Nonparametric estimation of multivariate convex-transformed
  densities.
\newblock {\em {\it Ann. Statist.} (to appear)}.
\newblock Available at: http://arxiv.org/abs/arXiv:0911.4151v1.

\bibitem[Song and Zhengjun, 2004]{chinos}
Song, W. and Zhengjun, C. (2004).
\newblock The generalized decomposition theorem in banach spaces and its
  applications.
\newblock {\em J. Approx. Theory}, 129:167--181.

\bibitem[Van~der Vaart and Wellner, 1996]{vw}
Van~der Vaart, A. and Wellner, J. (1996).
\newblock {\em Weak Convergence and Empirical Processes}.
\newblock Springer-Verlag, New York, NY, USA.

\bibitem[Varian, 1982a]{varian1}
Varian, H. (1982a).
\newblock The nonparametric approach to demand analysis.
\newblock {\em Econometrica}, 50:945--973.

\bibitem[Varian, 1982b]{varian2}
Varian, H. (1982b).
\newblock The nonparametric approach to production analysis.
\newblock {\em Econometrica}, 52:579--597.

\bibitem[Williams, 1991]{pwm}
Williams, D. (1991).
\newblock {\em Probability with Martingales}.
\newblock Cambridge University Press, Cambridge, UK.

\bibitem[Zhang, 2002]{zhang02}
Zhang, C.-H. (2002).
\newblock Risk bounds in isotonic regression.
\newblock {\em Ann. Statist.}, 30:528--555.

\end{thebibliography}
\bibliographystyle{apalike}

\end{document}